%% file: SRgroup2000.tex
\documentclass[11pt,a4paper]{amsart}
\usepackage[T1]{fontenc}
\usepackage[utf8]{inputenc}
\usepackage{amsmath,amssymb,amsthm,mathtools}
\usepackage[margin=27mm]{geometry}
\usepackage{booktabs,longtable,array,graphicx}
\usepackage{xcolor,listings}
\usepackage[hidelinks]{hyperref}
\usepackage{lmodern}
\usepackage{microtype}
\newtheorem{thm}{Theorem}[section]
\newtheorem{lemma}[thm]{Lemma}
\newtheorem{prop}[thm]{Proposition}
\newtheorem{corl}[thm]{Corollary}

\theoremstyle{definition}
\newtheorem{defn}[thm]{Definition}
\newtheorem{algorithm}{Algorithm}

\theoremstyle{remark}
\newtheorem{remark}[thm]{Remark}
\newcommand{\F}{\mathbb F}

\DeclareMathOperator{\Irr}{Irr}
\DeclareMathOperator{\GL}{GL}

\DeclareMathOperator{\Hom}{Hom}

\DeclareMathOperator{\rank}{rank}

\DeclareMathOperator{\Alt}{Alt}

\numberwithin{equation}{section}
\title[Classification of simply reducible groups]{A classification of finite simply reducible groups\\of order at most 2000}
\author[Y. Luan]{Yongzhi Luan}
\address{School of Mathematics, Shandong University, Jinan 250100, China}
\email{luanyongzhi@email.sdu.edu.cn}
\subjclass[2020]{20C15,20D15,20-04,15A63}
\keywords{simply reducible group, finite-group classification, multiplicity-free tensor product, special 2-group, quadratic map}
\date{\today}
\begin{document}
\raggedbottom
\begin{abstract}
We give a computer-assisted classification, up to isomorphism, of all
nontrivial finite simply reducible groups of order at most 2000. There
are 7889 such groups, all of even order. We provide representatives
and a table of their numbers at each order. Order 1024, which is not
available as a catalogue of group objects in the SmallGroups library of the computer algebra system GAP,
requires a separate construction. Combining descendant generation
with a classification of quadratic maps over $\mathbb F_2$, we obtain
exactly 803 groups of this order, including 221 of nilpotency class two.
We also derive exact counting formulas on several infinite families
of orders and explicit bounds from structural subclasses.
\end{abstract}
\maketitle
\setcounter{tocdepth}{1}
\tableofcontents
\clearpage

\section{Introduction}
Symmetry organizes the eigenspaces of a quantum-mechanical Hamiltonian
into representations of its symmetry group. When two systems carrying
irreducible representations $U$ and $V$ are coupled, the corresponding
representation is $U\otimes V$. Its irreducible decomposition describes
the possible symmetry types of the coupled states. Wigner introduced
simply reducible groups in 1941 in this setting, motivated by
symmetry-preserving perturbations of coupled eigenvalue problems
\cite{Wig41}; see also \cite[Introduction]{Luan20}.

The relevant simplification occurs when every irreducible constituent
of $U\otimes V$ appears at most once. For every irreducible $W$, each nonzero coupling
space $\Hom_G(W,U\otimes V)$ is then one-dimensional, so a normalized
intertwiner is determined up to a phase. Moreover, by Schur's lemma,
a $G$-invariant perturbation restricted to $U\otimes V$ acts as a scalar
on each irreducible summand: there is no additional multiplicity space
in which it can mix states of the same symmetry type. This is the
representation-theoretic reason that multiplicity-free coupling
simplifies the determination of symmetry-adapted states and tensor-operator
matrix elements. The familiar Clebsch--Gordan decomposition for angular
momentum illustrates this phenomenon. Wigner's definition combines
this condition with self-duality of the irreducible representations.

\begin{defn}[Wigner \cite{Wig41}]\label{def:SR}
A finite group $G$ is \emph{simply reducible}, or an \emph{SR-group}, if
\begin{enumerate}
\item every $g\in G$ is conjugate to $g^{-1}$; and
\item $\langle\chi\psi,\theta\rangle_G\in\{0,1\}$ for all
$\chi,\psi,\theta\in\Irr(G)$.
\end{enumerate}
Here representations are finite-dimensional over $\mathbb C$, and
$\langle a,b\rangle_G=|G|^{-1}\sum_{g\in G}a(g)\overline{b(g)}$.
\end{defn}
The first condition is called \emph{ambivalence}. It is equivalent to
all irreducible characters being real-valued, or to all irreducible
complex representations being self-dual. It does not require every
representation to be realizable over $\mathbb R$; quaternionic
representations are allowed.

The restriction to even orders follows immediately from ambivalence.
The identity is the only element of a finite group of odd order that
can be conjugate to its inverse
\cite[Chapter~XXVI, Exercise~1.1]{BKZ19}. Indeed, if
$xgx^{-1}=g^{-1}$, conjugation by $x$ induces on $\langle g\rangle$
an automorphism of odd order, whereas inversion has order at most two.
Thus $g=g^{-1}$, and odd order forces $g=1$. Consequently every
nontrivial finite SR-group has even order. The trivial group is SR
and is recorded separately below.

The two representation-theoretic conditions also impose strong
restrictions on the structure of the group. We shall use the following theorem in the arithmetic consequences
of the classification.
\begin{thm}[Kazarin--Chankov \cite{KC10}]\label{thm:sr-solvable}
Every finite simply reducible group is solvable.
\end{thm}
A classification without a bound on the order remains open
\cite[Problem~11.94]{KM26}. Examples and their automorphism groups
have been studied in \cite{Luan20,Luan21}. The aim of the present
paper is to complete the classification in the range $|G|\le2000$:
we construct one representative of each isomorphism type, justify
the exhaustiveness of the lists, and determine the number of types
at each order.

Write $f(n)$ for the number of isomorphism classes of SR-groups of
order $n$, with $f(1)=1$. Thus $f(n)$ counts the groups satisfying
Definition~\ref{def:SR}; it is not the number of all finite groups
of order $n$.
\begin{thm}[The classification]\label{thm:main-classification}
There are exactly $7889$ nontrivial SR-groups of order at most $2000$.
Of these, $803$ have order $1024$, distributed as
\[
\begin{array}{c|rrrrrrrrr}
\text{nilpotency class}&1&2&3&4&5&6&7&8&9\\\hline
\text{isomorphism classes}&1&221&373&133&48&16&7&2&2.
\end{array}
\]
Representatives and counts at all other even orders are given by the
SmallGroups identifiers in the accompanying catalogues. Representatives
at order $1024$ are given by reconstructible polycyclic presentations.
\end{thm}
The proof is computer-assisted. Section~\ref{sec:census} develops the
structural reductions and exact tests used to search the SmallGroups
catalogues at orders other than $1024$. Order $1024$ requires a separate
construction. Section~\ref{sec:1024higher} uses canonical descendants
to obtain $581$ groups of nilpotency class at least three.
Sections~\ref{sec:class2} and~\ref{sec:remaining} classify the class-two
groups through quadratic maps over $\mathbb F_2$, yielding $221$ types.
Together with the unique abelian SR-group of this order, these disjoint
lists give $f(1024)=803$. Exhaustive enumeration and exact isomorphism
classification establish the counts; numerical invariants serve as
checks and as descriptive information in the catalogues.

Section~\ref{sec:count-patterns} studies the arithmetic information
revealed by the classification. It proves formulas for $f(2m)$ with
$m$ odd, for $f(4p^e)$ with $p$ an odd prime, and for $f(8p)$ with
$p\ge5$ prime. It also gives structural counting formulas and explains
why a single multiplicative formula cannot describe all orders.
The complete count table appears in Appendix~\ref{sec:counttable}.

The classification catalogues, reconstruction data, and source code
are available in the accompanying GitHub repository \cite{github_luan26}.% \url{https://github.com/luanyongzhi/Simply-Reducible-Groups-2000}.
Pseudocode for the principal algorithms and instructions for reconstructing the
groups are provided in the appendices.

\section{The catalogue census: reductions, tests and evidence}
\label{sec:census}\label{sec:tools}
The catalogue search needs a test that rejects groups as soon as a
necessary condition fails and gives an exact answer for every survivor.
The structural results explain the search restrictions and the later
constructions; the character identities provide the exact acceptance
test used in the catalogue traversal. Closure also supplies the smaller-order
inputs for the separate construction at order 1024.

\subsection{Structural reductions}\label{subsec:structural-reductions}
The following closure properties justify passing from accepted groups
to quotients and constructing direct products. In particular, an SR group
can have only SR canonical parents, which will be the completeness
principle for the descendant search.
\begin{prop}[{\cite[Chapter~XXVI, Exercises~1.7--1.8]{BKZ19}}]\label{prop:closure}
Quotients and direct products of finite SR-groups are SR. Conversely, if
$A\times B$ is SR, both factors are SR. A central product of SR-groups is SR.
\end{prop}

Ambivalence gives inexpensive restrictions on the center and
abelianization. A group failing either restriction can be discarded
before the more detailed SR test. These restrictions also identify the
abelian and nilpotent parts of the census.
\begin{prop}[Elementary abelian restrictions]\label{prop:elementary}
For a finite ambivalent group $G$, both $Z(G)$ and $G/G'$ are elementary
abelian 2-groups, allowing the trivial group. The only abelian SR-groups
are elementary abelian 2-groups. Every nontrivial SR-group has even order,
and every finite nilpotent SR-group is a 2-group.
\end{prop}
\begin{proof}
A central element is conjugate only to itself, so $z=z^{-1}$.
The abelian quotient $G/G'$ is ambivalent and consequently has exponent
at most two. Conversely all irreducible characters of an elementary
abelian 2-group are linear and real, so that group is SR.

The odd-order assertion was established in the introduction.
A finite nilpotent group is the direct product of its Sylow subgroups;
an odd Sylow factor would be a nontrivial odd-order SR quotient.
\end{proof}

To record the contribution of direct products without mistaking it for
new groups, we also need an exact test for a central $C_2$ factor.
This test computes a classification flag; a positive flag is not a reason
to reject a group from the census.
\begin{lemma}[Detecting a central direct factor]\label{lem:c2-factor}
A finite $2$-group $G$ has a direct factor isomorphic to $C_2$ if and
only if $\Omega_1(Z(G))\nleq\Phi(G)$.
\end{lemma}
\begin{proof}
A central involution $z\notin\Phi(G)$ is avoided by some maximal
subgroup $M$. Since $|G:M|=2$ and $z$ is central,
$G=M\times\langle z\rangle$. Conversely, a generator of a direct
$C_2$ factor is a central involution avoided by the complementary
maximal subgroup, and hence does not belong to $\Phi(G)$.
\end{proof}

\subsection{An exact SR test}\label{subsec:exact-sr-test}
The structural restrictions are necessary but not sufficient. The final
acceptance decision uses the following character identity, which replaces
all tensor-product multiplicity checks by a single equality of integers.
For $g\in G$ write $r_2(g)=|\{x:x^2=g\}|$ and $c(g)=|C_G(g)|$.
\begin{thm}[Wigner {\cite[Theorem~2]{Wig41}}]\label{thm:wigner}
For every finite group,
\[
 \sum_{g\in G}r_2(g)^3\le\sum_{g\in G}c(g)^2,
\]
with equality if and only if $G$ is SR.
\end{thm}
The preliminary ambivalence check can be expressed in the same square-root
counts. We record the identity used by the descendant and export programs.
\begin{corl}[The quadratic moment test]\label{cor:ambivalence-moment}
For every finite group $G$,
\[
 \sum_{g\in G}r_2(g)^2
 =|G|\,|\{\chi\in\Irr(G):\chi=\overline\chi\}|.
\]
Consequently $G$ is ambivalent if and only if this sum equals $|G|k(G)$,
where $k(G)$ denotes the number of conjugacy classes.
\end{corl}
\begin{proof}
The Frobenius--Schur formula gives
$r_2=\sum_{\chi\in\Irr(G)}\varepsilon_\chi\chi$,
where $\varepsilon_\chi\in\{0,\pm1\}$ and
$\varepsilon_\chi^2=1$ precisely when $\chi$ is real-valued
\cite{Isa76}. Orthogonality gives the stated sum.
All irreducible characters are real-valued precisely when every
conjugacy class is fixed by inversion, because irreducible characters
separate conjugacy classes.
\end{proof}
Both tests use exact integers. If $C$ runs over the conjugacy classes,
then the right side of Wigner's equality can be computed as
$\sum_C |C|(|G|/|C|)^2$. The square map on classes computes the left
side without enumerating all triples: for each source class $C$, add
$|C|/|D|$ to the square-root count of the class $D$ containing $x^2$,
where $x\in C$. Conjugation-equivariance of squaring justifies this ratio.
Character-table checks provide an independent implementation of the definition.
The exact census filter is Algorithm~\ref{alg:srtest}.

\subsection{Catalogue traversal and computational evidence}\label{subsec:catalogue-evidence}
For every order $n\leq2000$ other than $1024$, the SmallGroups
library supplies a complete list of pairwise nonisomorphic groups
\cite{SmallGrp,BEO01,BEO2002}. We traverse these lists in GAP
\cite{GAP} and retain precisely the groups satisfying the exact SR test
of Theorem~\ref{thm:wigner}. The catalogue implementation first rejects any conjugacy class that is
not fixed by inversion, then applies the cubic-moment equality to the
survivors. Thus the necessary test precedes the exact acceptance test,
as in Subsection~\ref{subsec:exact-sr-test}.
Algorithm~\ref{alg:catalogue} gives the pseudocode corresponding to
\path{sr_search2_2000.g}; Algorithm~\ref{alg:srtest} gives the filter.
The identifier $(n,i)$ in every accepted row refers to
\texttt{SmallGroup(n,i)}. A structure description is supplied for
convenience; classification and reconstruction use the identifier,
not the description.

The two largest catalogue orders were computed separately:
there are $10\,494\,213$ groups of order $512$ and
$408\,641\,062$ groups of order $1536$. The scripts
\path{sr_search512.g} and \path{sr_search1536.g} apply the
same test to these complete lists. The catalogue computations were
carried out over three weeks using four remote servers and three
local computers. The resulting files have the following sizes.
\begin{center}\small
\begin{tabular}{@{}lrr@{}}\toprule
File & Orders & SR groups\\\midrule
\path{SR_groups_results2_2000.csv}
  & other even orders &6210\\
\path{SR_groups_results512.csv}&512&317\\
\path{SR_groups_results1536.csv}&1536&559\\\bottomrule
\end{tabular}
\end{center}
The three files have disjoint sets of order--identifier pairs.
Their union contains $7086$ isomorphism types at the $999$ even
orders other than $1024$. These records form the catalogue part of
Theorem~\ref{thm:main-classification}; the counts at individual orders
are listed in Appendix~\ref{sec:counttable}.

\medskip\noindent\textbf{Why order 1024 is treated separately.}
Order $1024$ is the sole exception to the availability of group
objects in this range. A numerical value of
\texttt{NumberSmallGroups(1024)} does not provide a list
\texttt{SmallGroup(1024,i)}; the library's coverage statement
explicitly excludes that order \cite{SmallGrp}. Consequently,
order $1024$ cannot be filled by the catalogue traversal. There are $49\,487\,367\,289$ groups of this order \cite{Burrell22}, so constructing every group and filtering afterwards would also be impractical.
We instead construct its SR-groups: canonical-parent generation
handles nilpotency class at least three, and quadratic maps handle
class two. The elementary abelian group accounts for class one.
These disjoint constructions are described in the following sections.

\section{Order 1024: canonical descendants of class at least three}
\label{sec:1024higher}
Let $P_0(G)=G$ and
$P_{i+1}(G)=P_i(G)^2[P_i(G),G]$ be the lower exponent-$2$
central series. If $c$ is the least index with $P_c(G)=1$, then
$c$ is the $2$-class of $G$, and its canonical parent is
$G/P_{c-1}(G)$. This quotient is intrinsic to $G$.
It is SR whenever $G$ is SR, by Proposition~\ref{prop:closure}.
We distinguish this series from the ordinary lower central series;
``class'' without a qualifier below means nilpotency class.

\begin{lemma}\label{lem:parent}
An immediate descendant of an elementary abelian $2$-group has
nilpotency class at most two. Every ambivalent $2$-group of
nilpotency class two has $2$-class two and an elementary abelian
canonical parent.
\end{lemma}
\begin{proof}
An elementary abelian parent has $2$-class one, so its immediate
descendant satisfies $P_2(G)=1$. Since $G'\leq P_1(G)$, this
implies $[G',G]=1$. Conversely, in an ambivalent group,
$G/G'$ and $Z(G)$ are elementary abelian
(Proposition~\ref{prop:elementary}). If its nilpotency class is
two, then $G^2\leq G'\leq Z(G)$, whence $P_1(G)=G'$ and
$P_2(G)=1$.
\end{proof}

It follows that all SR-groups of order $1024$ and class at least
three are obtained by taking the immediate descendants of all
nonabelian SR-groups of order $2^k$, for $3\leq k\leq9$, with
step size $s=10-k$. The step size specifies that the descendant
has order $2^s$ times the order of its parent. A nonabelian parent
has order at least $8$, which explains the range $1\leq s\leq7$.
No restriction to parents of order $512$ is justified: a last
nontrivial lower exponent-$2$ central term may have order larger
than two.

\subsection{Generation and completion records}
The program \path{scriptG.g} implements this search using
\texttt{PqDescendants} from ANUPQ \cite{ANUPQ}.
The default class bound is one above the parent's $2$-class, and
both capable and terminal descendants are included. No exponent,
metabelian or capability restriction is imposed. The
$p$-group generation algorithm constructs one representative of
each automorphism orbit of allowable subgroups of the
$p$-multiplicator; these orbits are exactly the isomorphism types
of immediate descendants \cite[Sections 2.1 and 2.3]{ANUPQ}.
Descendants belonging to different parent isomorphism types
cannot be isomorphic, since their canonical parents are intrinsic.
Thus this procedure provides exact isomorphism control without
using numerical invariants to identify groups.
Algorithm~\ref{alg:descendants} gives its pseudocode.

For each parent--step job, the program writes each accepted
presentation code as a record tagged \texttt{SRHIT} in its checkpoint,
and writes a
\texttt{DONE} line only after scanning the entire returned
descendant list. The final checkpoint
\path{scriptG_results.FINAL.bak} has precisely $539$ distinct
\texttt{DONE} records, one for every required job, and no failed
job. The accepted counts in those records agree exactly with the
$581$ \texttt{SRHIT} lines. The distribution is
\begin{center}\small
\begin{tabular}{@{}rrrrr@{}}\toprule
Parent order&Step&Jobs&Descendants tested&SR descendants\\\midrule
512&1&316&51\,886&474\\
256&2&129&798\,367&105\\
128&3&53&2\,317\,492&2\\
64&4&24&0&0\\
32&5&11&0&0\\
16&6&4&0&0\\
8&7&2&0&0\\\midrule
Total&&539&3\,167\,745&581\\\bottomrule
\end{tabular}
\end{center}
The completion check includes the $412$ jobs with no SR output,
as well as the $127$ jobs with at least one accepted descendant.
The console transcript \path{scriptG_console.FINAL.bak}
records the successful preliminary character-table comparisons
and the final summary. The complete checkpoint retains jobs from
earlier runs as well as the final resumed run.

\begin{prop}[Computer-assisted class-at-least-three census]
\label{prop:census1024}
There are exactly $581$ SR-groups of order $1024$ and nilpotency
class at least three. Their distribution by class is
\begin{center}
\begin{tabular}{@{}r*{7}{r}@{}}\toprule
Class&3&4&5&6&7&8&9\\\midrule
Number&373&133&48&16&7&2&2\\\bottomrule
\end{tabular}
\end{center}
\end{prop}
\begin{proof}
Quotient closure and Lemma~\ref{lem:parent} place every possible
group in the stated job roster. The complete catalogue lists of
smaller SR-groups supply all parents. The checked completion
records cover that roster exactly. The descendant algorithm's
orbit construction gives completeness and nonredundancy within
each job, and uniqueness of the canonical parent gives
nonredundancy between jobs. The exact SR filter leaves $581$
representatives, with the displayed class distribution.
\end{proof}

\subsection{Reconstructible classification records}
The program \path{makeCSV.g} reconstructs every accepted group
from its polycyclic presentation code, repeats the order and SR
tests, verifies that its class is at least three, and writes
\path{sr1024_class3plus.csv}.
Algorithm~\ref{alg:exportdescendants} describes this export.
All $581$ rows agree with the checkpoint's parent, step and code;
the group invariants have also been independently recomputed.
Each row records its parent order and catalogue identifier, step,
class, exponent, centre order, derived length, number of conjugacy
classes, a central $C_2$ direct-factor flag, and a presentation
code. These are explicit representatives: a row with code $a$
is reconstructed by \texttt{PcGroupCode(a,1024)}.
The code is a presentation encoding, not a canonical isomorphism
identifier; the generation argument above supplies the
nonisomorphism proof.
For example, the first row is recovered in GAP by
\begin{lstlisting}
a := Int("23429787501007350394706814058502769492323970612862");;
G := PcGroupCode(a,1024);;
[Size(G), NilpotencyClassOfGroup(G), Exponent(G)];
# [ 1024, 3, 8 ]
\end{lstlisting}
The second argument specifies the order of the encoded group
\cite[Section~46.9]{GAP}. The decimal code must be retained as an exact
integer. The supplied reader \path{read_sr1024_class3plus.g} reads
this column directly from the CSV and reconstructs any selected row;
its use is described in Appendix~\ref{app:implementation}.

Among these representatives, $221$ have a direct $C_2$ factor
and $360$ do not. The subtotal $474$ consists precisely
of the groups with canonical parent of order $512$; the
additional $107$ arise from the smaller canonical parents.

\section{Class-two groups and quadratic maps}\label{sec:class2}

We reduce the class-two stratum at order $2^{10}$ to a finite orbit
problem and give exact tests for the two conditions in the definition of
the SR conditions. The subsequent orbit computations complete this
stratum from the order-512 census.

\begin{lemma}[Central splitting]\label{lem:class2-split-new}
Let $G$ be a nonabelian SR $2$-group of nilpotency class two. Then
$G/G'$ and $Z(G)$ are elementary abelian, $\Phi(G)=G'$, and
\[
 G\cong S\times C_2^j,
 \qquad Z(S)=S'=\Phi(S),
\]
where $S$ is a class-two SR group. The isomorphism type of $S$ and the
integer $j$ are uniquely determined by $G$.
\end{lemma}
\begin{proof}
Ambivalence passes to $G/G'$, so this abelian quotient has exponent at
most two. A central element is conjugate only to itself, so every central
element has order at most two. Consequently $G^2\le G'$ and
$\Phi(G)=G^2G'=G'$. If $z\in Z(G)\setminus G'$, a maximal subgroup $K$
avoiding $z$ gives $G=K\times\langle z\rangle$. The factor $K$ is SR,
being a quotient of $G$. Iterate until the center equals the derived
subgroup. Uniqueness follows from the Krull--Remak--Schmidt theorem:
the remaining factor has no direct factor of order two, since its center
is contained in its Frattini subgroup.
\end{proof}

We call a group with $Z(S)=S'=\Phi(S)$ a \emph{special} $2$-group; in
this section it is also the stem factor of the central splitting.
Write
\[
 V=S/S'\cong\mathbb F_2^m,\qquad W=S'\cong\mathbb F_2^r.
\]
Commutators and squares define
\[
 \beta:\bigwedge^2V\longrightarrow W,
 \quad \beta(\bar x,\bar y)=[x,y],
 \qquad q:V\longrightarrow W,\quad q(\bar x)=x^2.
\]
Here and below $W$ is written additively. Since $W$ is central of
exponent two, these definitions are independent of the chosen lifts,
and
\begin{equation}\label{eq:polar-new}
 q(u+v)=q(u)+q(v)+\beta(u,v).
\end{equation}
The map $\beta$ is surjective as a linear map on $\bigwedge^2V$, and its
common radical is zero.

Conversely, the following construction recovers a group directly from
the linear-algebra data. Let $V=\mathbb F_2^m$ and $W=\mathbb F_2^r$,
and let $q:V\to W$ have polar map $\beta$. Choose a basis
$e_1,\ldots,e_m$ of $V$, and write $q_i=q(e_i)$ and
$b_{ij}=\beta(e_i,e_j)$ for $i<j$. Polarization gives
\[
 q(v)=\sum_i v_iq_i+\sum_{i<j}v_iv_jb_{ij}.
\]
Define a bilinear map $c:V\times V\to W$ by
\[
 c(u,v)=\sum_i u_iv_iq_i+\sum_{i<j}u_jv_i b_{ij},
\]
and equip the set $V\times W$ with the multiplication
\begin{equation}\label{eq:class2-construction}
 (u,w)(v,z)=(u+v,w+z+c(u,v)).
\end{equation}
We denote this group by $S(\beta,q)$. To check the construction,
bilinearity gives
\[
 c(u,v)+c(u+v,t)=c(v,t)+c(u,v+t),
\]
which is precisely associativity of
\eqref{eq:class2-construction}. The identity is $(0,0)$, the inverse
of $(v,w)$ is $(v,w+q(v))$, and direct calculation yields
\[
 (v,w)^2=(0,q(v)),\qquad
 [(u,w),(v,z)]=(0,\beta(u,v)).
\]
Consequently, if $\beta:\bigwedge^2V\to W$ is surjective and has
zero common radical, then
\[
 |S(\beta,q)|=2^{m+r},\qquad
 S(\beta,q)'=Z(S(\beta,q))=\{0\}\times W.
\]
All squares lie in this derived subgroup, so the Frattini subgroup
is also $\{0\}\times W$. Thus the construction produces a special
group with exactly the prescribed square and commutator maps.

\begin{prop}[Quadratic classification dictionary]
\label{prop:class2-dictionary-new}
Special $2$-groups with elementary abelian center and invariants $(m,r)$
correspond to $\mathrm{GL}(V)\times\mathrm{GL}(W)$-orbits of quadratic
maps $q:V\to W$ whose polar map is surjective and has zero common
radical. A pair $(A,B)$ identifies $q$ and $q'$ when
$q'(Av)=Bq(v)$ for every $v\in V$.
\end{prop}
\begin{proof}
The construction above gives a group for every admissible pair
$(\beta,q)$. Equivalently, if $z_1,\ldots,z_r$ is a basis of $W$,
take central generators $z_j$ of order two and generators
$x_1,\ldots,x_m$ with
\[
 x_i^2=q_i,\qquad [x_i,x_j]=b_{ij}\quad(i<j),
\]
where the right-hand sides denote the corresponding products of the
$z_j$. Collection gives at most $2^{m+r}$ normal forms
$x_1^{a_1}\cdots x_m^{a_m}z_1^{d_1}\cdots z_r^{d_r}$ with binary
exponents. The model \eqref{eq:class2-construction} realizes all these
forms distinctly, so the presentation has exactly these elements.

For an arbitrary special group with elementary abelian center,
choose lifts of a basis of its central quotient. They satisfy this
presentation for the extracted $(\beta,q)$ and generate the group
together with its center. Since the orders agree, the resulting
surjective homomorphism from $S(\beta,q)$ is an isomorphism.

Finally, an isomorphism induces linear maps $A$ on the central quotient
and $B$ on the derived subgroup, and preservation of squares gives
$q'(Av)=Bq(v)$. Conversely, this identity also preserves the polar maps.
Send each $x_i$ to a lift of $Ae_i$ and act on the central generators
by $B$. The square and commutator relations are preserved, so the
presentation gives an isomorphism. Different choices of lifts do not
change these relations because the center has exponent two.
\end{proof}

Quadratic descriptions of special $2$-groups and their character theory
are developed in \cite{KK15}. We provide the arguments needed for the
tensor-product test explicitly.

For $\lambda\in W^*=\operatorname{Hom}(W,\mathbb F_2)$, put
\[
 B_\lambda=\lambda\circ\beta,\quad
 R_\lambda=\operatorname{rad}(B_\lambda),\quad
 \rho_\lambda=\operatorname{rank}(B_\lambda).
\]
Thus $R_0=V$ and $\rho_0=0$. Let $H_\lambda$ be the inverse image of
$R_\lambda$ in $S$.

\begin{lemma}[Character supports]\label{lem:class2-support-new}
An irreducible character $\chi$ with central character
$w\mapsto(-1)^{\lambda(w)}$ has degree
$d_\lambda=2^{\rho_\lambda/2}$. It vanishes outside $H_\lambda$, and
on $H_\lambda$ it equals $d_\lambda\theta_\chi$ for a linear character
$\theta_\chi$ of $H_\lambda$. There are exactly
$2^{m-\rho_\lambda}$ irreducible characters above $\lambda$.
\end{lemma}
\begin{proof}
If $g\notin H_\lambda$, some $h$ satisfies
$(-1)^{\lambda([g,h])}=-1$. Conjugation by $h$ therefore gives
$\chi(g)=-\chi(g)$, so $\chi(g)=0$. If $g\in H_\lambda$, its
representation matrix commutes with every representation matrix;
Schur's lemma makes it scalar. These scalars define
$\theta_\chi$. Character orthogonality now gives
\[
 1=\frac{|H_\lambda|d_\lambda^2}{|S|},
 \qquad d_\lambda^2=2^{\rho_\lambda}.
\]
The sum of squared degrees above a fixed central character is
$|S:W|=2^m$, for example by decomposing the corresponding central
idempotent summand of $\mathbb C[S]$. This yields the stated number.
\end{proof}

\begin{thm}[Exact SR test]\label{thm:class2-SR-new}
The special group $S(\beta,q)$ constructed in
\eqref{eq:class2-construction} is SR if and only if both
of the following conditions hold:
\begin{align}
 &\lambda(q(v))=0
   &&(\lambda\in W^*,\ v\in R_\lambda),\label{eq:class2-R-new}\\
 &\rho_\lambda+\rho_\mu+\rho_{\lambda+\mu}
   =2\bigl(m-\dim(R_\lambda\cap R_\mu)\bigr)
   &&(0\ne\lambda\ne\mu\ne0).\label{eq:class2-M-new}
\end{align}
The second condition is equivalent to multiplicity-freeness and depends
only on $\beta$. For a fixed $\beta$, the first condition is an affine
linear system in the $mr$ coordinates of $q$ on a basis of $V$.
\end{thm}
\begin{proof}
For $g$ lying over $v$, its conjugacy class is
$g\beta(v,V)$ and $g^{-1}=gq(v)$. Hence ambivalence is equivalent to
$q(v)\in\beta(v,V)$. The annihilator of $\beta(v,V)$ consists of the
$\lambda$ with $v\in R_\lambda$, proving
\eqref{eq:class2-R-new}. Since $\lambda q$ restricts to a linear
function on $R_\lambda$, it suffices to impose this condition on a
basis of each radical. Any two refinements of $\beta$ differ by a
linear map $V\to W$, so these are affine linear equations.

Now take $\chi,\psi,\eta$ above $\lambda,\mu,\nu$, respectively.
A nonzero coefficient $\langle\chi\psi,\eta\rangle$ requires
$\nu=\lambda+\mu$. Put
$T=R_\lambda\cap R_\mu$ and $t=\dim T$. Then
$T\le R_\nu$, and Lemma~\ref{lem:class2-support-new} gives
\[
 \langle\chi\psi,\eta\rangle
 =\frac{d_\lambda d_\mu d_\nu}{|S|}
  \sum_{g\in\pi^{-1}(T)}
  \theta_\chi(g)\theta_\psi(g)\overline{\theta_\eta(g)}.
\]
The summand is a linear character of $\pi^{-1}(T)$, trivial on $W$.
Its sum is either zero or $|W|2^t$. Consequently every nonzero
multiplicity in this product is
\begin{equation}\label{eq:class2-multiplicity-new}
 2^{(\rho_\lambda+\rho_\mu+\rho_\nu)/2+t-m}.
\end{equation}
At least one constituent occurs, so the exponent is a nonnegative
integer. Thus all products are multiplicity-free exactly when these
exponents vanish. The cases $\lambda=0$, $\mu=0$, and
$\lambda=\mu$ give zero exponents automatically; all other cases
are \eqref{eq:class2-M-new}.
\end{proof}

One may check \eqref{eq:class2-M-new} once per two-dimensional subspace
of $W^*$: its three pairs have the same rank sum and common radical.
There are $(2^r-1)(2^r-2)/6$ such checks.
Algorithm~\ref{alg:class2-filter} gives the exact filter and affine-refinement procedure.

\begin{thm}[Stem pencils]\label{thm:class2-pencil-new}
Suppose $r=2$ and write $B_1,B_2,B_{12}=B_1+B_2$ for the nonzero
forms of the commutator pencil. The pencil underlies an SR group if
and only if
\begin{equation}\label{eq:pencil-ranks-new}
 \rho_1+\rho_2+\rho_{12}=2m.
\end{equation}
Every stem pencil admits an ambivalent refinement. If
\eqref{eq:pencil-ranks-new} holds, it has exactly $2^m$ ambivalent
refinements before taking isomorphism orbits.
\end{thm}
\begin{proof}
The intersection of any two distinct radicals is the common radical,
which is zero. Theorem~\ref{thm:class2-SR-new} therefore reduces
multiplicity-freeness to \eqref{eq:pencil-ranks-new}.

Choose a refinement $(q_1^0,q_2^0)$ and seek linear corrections
$(\ell_1,\ell_2)\in V^*\oplus V^*$. The reality equations prescribe
$\ell_1$ on $R_1$, $\ell_2$ on $R_2$, and $\ell_1+\ell_2$ on
$R_{12}$. Their row spaces in $V\oplus V$ are
\[
 (R_1,0),\qquad (0,R_2),\qquad
 \{(v,v):v\in R_{12}\}.
\]
A dependence among these requires a common vector in
$R_1\cap R_2\cap R_{12}=0$. Thus the equations are independent,
and any prescribed right-hand side is attainable. The number of
solutions is
\[
 2^{2m-\dim R_1-\dim R_2-\dim R_{12}}
 =2^{\rho_1+\rho_2+\rho_{12}-m},
\]
which equals $2^m$ in the multiplicity-free case.
\end{proof}

At $(m,r)=(8,2)$, the possible rank multisets are therefore exactly
\[
 \{2,6,8\},\qquad \{4,4,8\},\qquad \{4,6,6\}.
\]
Each is realized by an orthogonal sum of scalar symplectic blocks
taking values in the three nonzero lines of $W$, with half-dimensions
$(3,1,0)$, $(2,2,0)$, and $(2,1,1)$, respectively. A choice of scalar
quadratic refinements gives a central quotient of a direct product
of extraspecial groups, hence an SR group. This establishes existence
of the profiles, not a count of the isomorphism classes within them.

\subsection{Four commutator pencils at \texorpdfstring{$(m,r)=(8,2)$}{(m,r)=(8,2)}}

Let $H_w$ denote a two-dimensional block on which the alternating
$W$-valued form is $w$ times the standard symplectic form,
where $0\ne w\in W=\mathbb F_2^2$. Let $K$ denote the
three-dimensional block
\[
 B_1=t\wedge x,\qquad B_2=t\wedge y.
\]
Orthogonal sums below are simultaneous orthogonal sums for both forms.

\begin{prop}[Multiplicity-free pencils]
\label{prop:MF-pencil-blocks-new}
A stem alternating pencil over $\mathbb F_2$ satisfies
\eqref{eq:pencil-ranks-new} if and only if it is an orthogonal sum
of blocks $H_w$ and $K$. Its orbit under
$\mathrm{GL}(V)\times\mathrm{GL}(W)$ is determined by the number
$k$ of $K$ blocks and the unordered triple $(a,b,c)$ of multiplicities
of the three colors $H_w$. In particular there are exactly four
such pencil orbits when $m=8$:
\[
\begin{array}{c|c|c}
 (a,b,c)&k&\{\rho_1,\rho_2,\rho_{12}\}\\\hline
 (3,1,0)&0&\{2,6,8\}\\
 (2,2,0)&0&\{4,4,8\}\\
 (2,1,1)&0&\{4,6,6\}\\
 (1,0,0)&2&\{4,6,6\}
\end{array}
\]
Each of these four pencils has exactly $256$ ambivalent refinements
before the action of its stabilizer is taken.
\end{prop}
\begin{proof}
We use the canonical decomposition of a pair of alternating forms,
which is valid in characteristic two; see \cite{Scharlau76}, \cite[Theorem~4.1]{DD18}
for its singular part and \cite[pp.~57--59]{FV04} for the associated
Kronecker modules. A singular block has dimension $2h+1$ and all
three nonzero forms have rank $2h$. A regular primary block is the
alternating doubling of an elementary-divisor block; if its elementary
divisor has degree $n$, its dimension is $2n$. These decompositions
are unique up to permutation and equivalence of blocks.

The defect
\[
 D=\rho_1+\rho_2+\rho_{12}-2m
\]
is additive on orthogonal sums. A singular block has defect $2h-2$.
Stemness excludes $h=0$, since that block is a common radical.
Thus the only singular blocks of defect zero have $h=1$ and are $K$.

For a regular block with a linear elementary divisor of exponent
$n$, two of the three scalar forms have full rank $2n$, while the
third has rank $2n-2$. Its defect is $2n-2$, which vanishes only
when $n=1$. These three rational projective eigenvalues give
precisely the three colors $H_w$. If the elementary divisor has an
irreducible factor of degree at least two, none of the three
rational scalar combinations is singular; all three ranks equal
the block dimension, and its defect is strictly positive.

Every block therefore has nonnegative defect. The total defect
vanishes exactly for sums of $H_w$ and $K$. Under
$\mathrm{GL}_2(\mathbb F_2)$ the three colors are permuted
arbitrarily, and the singular block type is preserved. The canonical
invariants consequently give $(k,\{a,b,c\})$ as stated.

For $m=8$, the dimension equation is
$2(a+b+c)+3k=8$. It gives $k=0$ with $a+b+c=4$, or $k=2$ with
$a+b+c=1$. In the first case the partition $(4,0,0)$ is excluded
because the pencil would span only one scalar direction. The
remaining choices are exactly those in the table. Their ranks follow directly from the blocks. We spell out the
refinement count in this dimension. Fix any refinement
$q^0=(q_1^0,q_2^0)$ of one of the four pencils. Every other
refinement has the unique form
\[
 q=q^0+(\ell_1,\ell_2),\qquad
 (\ell_1,\ell_2)\in V^*\oplus V^*,
\]
so there are initially $2^{16}$ choices. Ambivalence imposes the
equations
\[
 \ell_1|_{R_1}=q_1^0|_{R_1},\qquad
 \ell_2|_{R_2}=q_2^0|_{R_2},\qquad
 (\ell_1+\ell_2)|_{R_{12}}=(q_1^0+q_2^0)|_{R_{12}}.
\]
The right-hand sides are linear on their respective radicals.
As in the proof of Theorem~\ref{thm:class2-pencil-new}, the three
families of equations have independent row spaces: a dependence
would give a vector in $R_1\cap R_2\cap R_{12}$, which is zero by
stemness. For each row of the table their total rank is
\[
 \dim R_1+\dim R_2+\dim R_{12}
 =3\cdot8-(\rho_1+\rho_2+\rho_{12})=24-16=8.
\]
Thus the affine system is consistent and has dimension $16-8=8$.
It has exactly $2^8=256$ solutions. These are refinements of a fixed
pencil; the subsequent stabilizer orbits, rather than the $256$
individual solutions, count group isomorphism types.
\end{proof}

This gives an exhaustive and small input for the $(8,2)$ computation:
four explicit pairs of alternating $8\times8$ matrices and $1024$
quadratic refinements in total. The last two rows show why an unordered
rank profile alone does not identify a commutator pencil.

\begin{thm}[The complete $(8,2)$ stratum]
\label{thm:pencil-eighteen-new}
There are exactly $18$ isomorphism classes of SR groups $S$ of order
$1024$ with $Z(S)=S'\cong C_2^2$. Their quotient by the center is
$C_2^8$. The four commutator pencils in
Proposition~\ref{prop:MF-pencil-blocks-new} carry respectively
$4$, $3$, $6$, and $5$ isomorphism classes of quadratic refinements.
\end{thm}
\begin{proof}
The four commutator pencils are exhaustive by
Proposition~\ref{prop:MF-pencil-blocks-new}. No groups belonging to
different pencil orbits can be isomorphic, since a group isomorphism
preserves the commutator map. We count the refinement orbits for
each pencil.

For a pencil made entirely of $H_w$ blocks, combine the blocks with
the same nonzero $w\in W$ into a symplectic space $V_w$. Thus
$V=\bigoplus_w V_w$, omitting zero-dimensional summands. If
$\lambda$ is the unique nonzero element of $W^*$ annihilating $w$,
then
\[
 R_\lambda=V_w.
\]
Consequently the summands are intrinsic, up to a permutation of
colors induced by $\mathrm{GL}(W)$. Reality forces
$q(V_w)\subseteq\langle w\rangle$: for a nonzero $v\in V_w$,
the commutator image $\beta(v,V)$ is exactly this line. Moreover,
orthogonality and polarization give
$q(\sum_wv_w)=\sum_wq(v_w)$. Thus a refinement is precisely a
choice of a nonsingular scalar quadratic form on each $V_w$.

Each such form has two isometry types, distinguished by its Arf
invariant. The only possible identifications permute colors whose
symplectic dimensions agree. For $(a,b,c)=(3,1,0)$ the two nonzero
dimensions differ, giving $2\cdot2=4$ choices. For $(2,2,0)$ the
two signs can be interchanged, giving three unordered choices.
For $(2,1,1)$ the sign on the dimension-four summand is independent
of the unordered pair of signs on the two dimension-two summands,
giving $2\cdot3=6$ choices. The required permutations and scalar
isometries exist, so these are both upper and lower bounds.

It remains to count refinements of $K\perp K\perp H_{w_1}$.
Take a basis
\[
 a_1,b_1,c_1,a_2,b_2,c_2,p,q
\]
of $V$, with scalar polar forms
\[
 B_1=a_1^*\wedge b_1^*+a_2^*\wedge b_2^*+p^*\wedge q^*,
 \qquad
 B_2=a_1^*\wedge c_1^*+a_2^*\wedge c_2^*.
\]
Here $B_2$ is the unique rank-four form, whereas $B_1$ and
$B_1+B_2$ have rank six. Their radicals are
\[
 R_1=\langle c_1,c_2\rangle,\quad
 R_2=\langle b_1,b_2,p,q\rangle,\quad
 R_{12}=\langle b_1+c_1,b_2+c_2\rangle.
\]
The subspace
\begin{equation}\label{eq:lastpencil-X-new}
 X=R_2\cap(R_1+R_{12})=\langle b_1,b_2\rangle
\end{equation}
is intrinsic to the pencil: $R_2$ is determined by the unique
rank-four form, and interchanging the two rank-six forms leaves
their radical sum unchanged. Hence the Boolean condition
$q|_X=0$, together with the number of involutions, is an invariant
of the full $\mathrm{GL}(V)\times\mathrm{GL}(W)$-orbit.

There are exactly $256$ reality-compatible refinements by
Theorem~\ref{thm:class2-pencil-new}. The accompanying exact GAP computation in Algorithm~\ref{alg:last-pencil} constructs
$15$ explicit pairs of invertible matrices, checks for each pair
the full tensor identity preserving $(B_1,B_2)$, and computes their
orbits on all $256$ refinements. The output is
\[
\begin{array}{c|c|c}
 \text{orbit size}&q|_X=0&\text{number of involutions}\\\hline
 48&\text{yes}&383\\
 48&\text{no}&383\\
 48&\text{no}&255\\
 96&\text{no}&191\\
 16&\text{yes}&255
\end{array}
\]
The sizes sum to $256$. Thus the verified tensor automorphisms give
at most five isomorphism classes. Conversely, the five rows have
distinct pairs of intrinsic invariants, so none can merge under
any additional pencil automorphism. This proves that there are
exactly five full orbits; it is unnecessary to assert that the
$15$ matrices generate the entire pencil stabilizer.

The total is therefore $4+3+6+5=18$.
\end{proof}

The computation in the last part is an orbit certificate with an
explicit exhaustiveness argument: its domain is the entire affine
space of $256$ solutions of the reality equations, its generators
are checked against the polar tensors, and its five resulting
orbits are separated by intrinsic invariants. The corresponding implementation and orbit table are identified in
Appendix~\ref{app:implementation} and include reconstructible group
representatives. This completes the $(8,2)$ branch. The branches with larger derived
subgroup are treated by quotient extensions and graph orbits below.

\subsection{Constant-rank-two spaces and the star family}

\begin{lemma}[Rank-two dichotomy over $\mathbb F_2$]
\label{lem:rank2-dichotomy-new}
Let $\mathcal N\le\bigwedge^2V^*$ be a subspace all of whose nonzero
forms have rank two. Then either
\[
 \mathcal N\le\varphi\wedge V^*
 \quad\text{for some }0\ne\varphi\in V^*,
\]
or $\mathcal N\le\bigwedge^2U$ for a three-dimensional
subspace $U\le V^*$.
\end{lemma}
\begin{proof}
A nonzero rank-two form is decomposable and determines a two-plane in
$V^*$. The sum of two distinct such forms has rank two precisely when
their two-planes intersect in a line. Consider the resulting family
of pairwise intersecting two-planes. If they share a line, the first
alternative holds. Otherwise choose two planes meeting in a line
$L$, and a third not containing $L$. These three planes lie in their
common three-dimensional span $U$. Any further plane not contained
in $U$ would meet $U$ in at most a line, and that line would have to
belong to all three chosen planes, a contradiction. This gives the
second alternative. The one- and two-dimensional cases are included
in the first alternative.
\end{proof}

\begin{lemma}[A reality obstruction]\label{lem:wedge3-obstruction-new}
If a commutator space contains $\bigwedge^2U$ for some
three-dimensional $U\le V^*$, it admits no ambivalent quadratic
refinement.
\end{lemma}
\begin{proof}
Choose independent coordinate functions $x,y,z$ spanning $U$ and
scalar components with polar forms $x\wedge y$, $x\wedge z$,
$y\wedge z$. On a coordinate complement to their common kernel,
the corresponding refinements have the form
$q_{xy}=xy+\ell_{xy}$, $q_{xz}=xz+\ell_{xz}$,
$q_{yz}=yz+\ell_{yz}$, where the $\ell$ are linear. At a vector
with coordinates $(x,y,z)$, that vector lies in the radical of
$zB_{xy}+yB_{xz}+xB_{yz}$. Reality would therefore require
\[
 zq_{xy}+yq_{xz}+xq_{yz}=0
 \qquad\text{on }\mathbb F_2^3.
\]
The left side has square-free cubic coefficient one, whereas its
remaining terms have degree at most two. Uniqueness of the
square-free polynomial representation of functions on
$\mathbb F_2^3$ gives a contradiction.
\end{proof}

\begin{prop}[The star family]\label{prop:star-new}
Let $m\ge3$, $U=\mathbb F_2^{m-1}$, $V=\mathbb F_2\oplus U$, and
$W=U$. For
\[
 \beta((t,u),(s,v))=tv+su,
\]
the ambivalent refinements are precisely
\[
 q(t,u)=tu+\delta u+tw,
 \qquad \delta\in\mathbb F_2,\quad w\in U.
\]
They form exactly three isomorphism orbits, and all are SR. Their
involution counts are
\[
 2^{m-1}(2^{m-1}+1)-1\quad\text{twice},
 \qquad 2^{m-1}-1\quad\text{once}.
\]
Every stem constant-rank-two SR group is one of these groups, or
$D_8$ or $Q_8$ when $m=2$.
\end{prop}
\begin{proof}
For $t=0$, reality requires $q(0,u)\in\langle u\rangle$ for every
$u$. A linear map on a space of dimension at least two with this
property is a scalar map, giving $q(0,u)=\delta u$. For $t=1$ the
commutator image is all of $W$, so there are no further restrictions.
The stabilizer of the star space acts by
$(t,u)\mapsto(t,Au+tb)$, with induced action $A$ on $W$; it acts
on the parameters by
\[
 \delta\mapsto\delta,
 \qquad w\mapsto Aw+(\delta+1)b.
\]
This gives one orbit for $\delta=0$ and two for $\delta=1$, according
as $w=0$ or $w\ne0$. Distinct scalar polar forms have rank two and
radical intersection dimension $m-3$, so the multiplicity equation
is automatic. The zero counts of $q$ are $2^{m-1}+1$ twice and one
once; multiplying by $|W|$ and subtracting one gives the involution
counts.

For the final assertion apply Lemma~\ref{lem:rank2-dichotomy-new}.
In the first alternative a dimension-$r$ space has common radical
of dimension at least $m-r-1$, with equality for independent
factors modulo $\langle\varphi\rangle$. Zero common radical thus
forces the full star space and $r=m-1$. In the second alternative,
zero common radical forces $m\le3$. The only case not already
of star type is $m=r=3$, excluded by
Lemma~\ref{lem:wedge3-obstruction-new}. If $m=2$, the scalar
nondegenerate form gives $D_8$ and $Q_8$.
\end{proof}

In particular, a stem constant-rank-two SR group has order an odd
power of two. This statement does not apply to arbitrary class-two
SR groups: adjoining elementary abelian direct factors gives such
groups also at orders $256$ and $1024$.

\subsection{Invariants and reconstructible representatives}\label{sec:class2-invariants}
The quadratic data also give the invariant columns of the catalogue.
For a special group associated with $q:V\to W$, every element over
$v$ has square $q(v)$, independently of its $W$-coordinate. Hence
\begin{equation}\label{eq:quadratic-involutions}
 \#\{g\in S:|g|=2\}=2^r|q^{-1}(0)|-1.
\end{equation}
For each $v\in V$, the conjugacy classes in the fibre over $v$ are
cosets of $\beta(v,V)$ in $W$. Therefore
\begin{equation}\label{eq:quadratic-classes}
 k(S)=\sum_{v\in V}2^{r-\dim\beta(v,V)}.
\end{equation}
The presentation in Proposition~\ref{prop:class2-dictionary-new}
constructs an actual group representative from each retained quadratic
orbit. Its order, centre, derived subgroup, involution count and class
number can then be checked against the data. These checks validate
representatives; the orbit classification proves their nonisomorphism.
Algorithm~\ref{alg:class2-export} describes the catalogue export.

\section{Completing the class-two census at order 1024}\label{sec:remaining}
The purpose of the following computations is to turn the quadratic-map
reduction into an explicit list of groups. After the $95$ non-special
classes have been obtained by central direct factors, the special
stratum has only five possible dimension pairs. We exclude $(4,6)$
and $(5,5)$; the pencil theorem gives $18$ classes at $(8,2)$.
Quotient extensions and exact graph orbits give $28$ and $80$ classes
at $(6,4)$ and $(7,3)$. Thus there are $126$ special groups and $221$
class-two groups. The output is the catalogue
\path{sr1024_class2.csv}, with reconstructible representatives.

\subsection{Separating central direct factors}
The central splitting lemma gives a bijection
\[
 H\longmapsto H\times C_2
\]
from class-two SR groups of order 512 to non-special class-two SR
groups of order 1024 (in both cases up to isomorphism).
Surjectivity follows by splitting off one central involution outside
the derived subgroup; injectivity is cancellation in the
Krull--Remak--Schmidt theorem. This is not an assertion that every special
group is directly indecomposable. For example $D_8\times D_8$ is special.

The class-two census at smaller orders, with structural checks on its identifiers,
gives
\begin{center}
\begin{tabular}{@{}r*{7}{r}@{}}\toprule
Order&8&16&32&64&128&256&512\\\midrule
Class two&2&2&7&10&20&42&95\\
Special&2&0&5&3&10&22&53\\\bottomrule
\end{tabular}
\end{center}
The difference of consecutive class-two counts equals the special count
at the larger order. Combining this with the class-at-least-three
census of Proposition~\ref{prop:census1024} gives
\begin{equation}\label{eq:final-count}
 N_2(1024)=95+s_{10},\qquad f(1024)=1+(95+s_{10})+581=677+s_{10}.
\end{equation}
Here $s_{10}$ counts all special SR groups of order 1024, including
direct products of nonabelian special factors. It is not a count of
groups with centre of order two: such a special group would have
$\dim(G/G')=9$, impossible for a nondegenerate scalar alternating form.

\subsection{The dimension pairs and two exclusions}
For a special target $S$, the dimensions
$m=\dim(S/S')=d(S)$ and $r=\dim S'$ measure its generator number and
derived-subgroup rank. Since $|S|=2^{m+r}=1024$, we have $m+r=10$.
The group is nonabelian, so $r\ge1$. Its commutators span $S'$, giving
$r\le\binom m2$; thus $m\ge4$ and $m\le9$.
For $r=1$, zero common radical means that the scalar alternating form
is nondegenerate, forcing $m$ even. This excludes $(9,1)$.
The entire special search is consequently partitioned into
\[
 (m,r)=(4,6),(5,5),(6,4),(7,3),(8,2).
\]
This is an exhaustive partition by intrinsic group invariants, so counts
from different pairs can be added without an isomorphism comparison.

\begin{prop}\label{prop:46-exclusion}
There is no special SR-group with $(m,r)=(4,6)$.
\end{prop}
\begin{proof}
Its commutator space would be all of $\Alt(\F_2^4)$. Write
$e_{ij}=e_i^*\wedge e_j^*$. The forms
$A=e_{12}+e_{34}$ and $B=e_{13}+e_{24}$ have ranks four and four,
their sum has rank two, and their radicals intersect trivially.
The multiplicity criterion would require $4+4+2=2\cdot4$, a contradiction.
\end{proof}

\begin{lemma}[Full-star obstruction]\label{lem:star-extension}
Let $V=\F_2\oplus U$, $\dim U\ge2$. No proper extension of the full
star commutator space $t\wedge U^*$ admits an ambivalent quadratic
refinement.
\end{lemma}
\begin{proof}
The star components of an ambivalent refinement have the form
$q_{\rm st}(t,u)=tu+\delta u+tw$, by the star-family calculation.
An additional scalar alternating form, after adding a star form, is
a nonzero alternating form $\omega$ on $U$. Its quadratic refinement
has the form $Q(u)+ct+\ell(u)$ with polar form $\omega$.
For each $u$ set $a=\omega(u,-)$. The vector $(1,u)$ is in the radical
of $\omega+t\wedge a$. Reality at this vector forces
\[
 Q(u)+c+\ell(u)+\omega(u,w)=0\qquad(u\in U).
\]
At $u=0$ we get $c=0$; the remaining identity makes $Q$ linear,
contradicting its nonzero polar form.
\end{proof}

\begin{prop}[Finite linear-algebra certificate]\label{prop:55-exclusion}
There is no special SR-group with $(m,r)=(5,5)$.
\end{prop}
\begin{proof}
Every nonzero alternating form on $\F_2^5$ has rank two or four.
A multiplicity-free space of dimension at least two contains a
rank-two form: otherwise the sum of ranks for any independent pair
would be twelve, exceeding twice the ambient dimension. Since
$\GL_5(2)$ is transitive on rank-two forms, normalize one to $e_{12}$.

Spaces of dimension five containing $e_{12}$ correspond bijectively
to four-dimensional subspaces of the nine-dimensional quotient
$\Alt(\F_2^5)/\langle e_{12}\rangle$. There are
$\binom94_2=3\,309\,747$ such spaces. The accompanying exact
certificate in Algorithm~\ref{alg:small-exclusion} enumerates their unique reduced
row echelon descriptions, rejecting an entire branch as soon as the
multiplicity equation fails for one triple. It accounts for all
completed subspaces by an exact branch-size sum. Precisely 84 spaces
survive this filter. Each contains a full-star hyperplane, as verified
by testing the common wedge factor of its fifteen rank-two forms.
Lemma~\ref{lem:star-extension} excludes all 84. Independently, solving
the affine reality system for each survivor returns inconsistency.
\end{proof}

The certificate includes the following control cases. These are
coordinate-subspace counts, not numbers of group isomorphism types.
\begin{center}\small
\begin{tabular}{@{}ccrrr@{}}\toprule
$(m,r)$&Normalization&Subspaces&MF&MF and real refinement\\\midrule
$(3,3)$&none&1&1&0\\
$(4,4)$&none&651&0&0\\
$(5,4)$&contains $e_{12}$&788035&1691&3\\
$(5,5)$&contains $e_{12}$&3309747&84&0\\\bottomrule
\end{tabular}
\end{center}
The three coordinate spaces in the $(5,4)$ control case are three
full stars in a single $\GL_5(2)$-orbit. Each has 32 ambivalent
refinements; the three group types arise only from the subsequent
stabilizer action on those refinements.

It follows that
\begin{equation}\label{eq:remaining-three}
 s_{10}=s_{6,4}+s_{7,3}+18.
\end{equation}
where $s_{m,r}$ counts special SR groups with these dimensions.
The pencil theorem completes the $(8,2)$ branch. We now enumerate
the admissible systems of dimensions $(6,4)$ and $(7,3)$.

\subsection{Quotient extensions and exact graph orbits}
The key reduction is to extend the polar spaces of known quotients,
rather than enumerate a full Grassmannian or cohomology space.

\begin{thm}[Coverage by polar extensions of order-$512$ quotients]
\label{thm:quotient-polar-complete}
Fix $m+r=10$ with $r\ge2$, and suppose that $\mathcal P$ contains
one representative of every class-two SR group of order $512$ with
$\dim(P/P')=m$. For $P\in\mathcal P$, extract its scalar commutator
space
\[
 \mathcal N_P\le\Alt(P/P'),\qquad \dim\mathcal N_P=r-1.
\]
Choose coordinates $P/P'\cong\mathbb F_2^m$ separately for each parent.
Then every special SR group of order $1024$ with dimensions $(m,r)$
is represented by a quadratic refinement of a space
\[
 \mathcal N=\mathcal N_P+\langle B\rangle,
 \qquad 0\ne B+\mathcal N_P\in\Alt(\mathbb F_2^m)/\mathcal N_P,
\]
for some $P\in\mathcal P$. It is sufficient to enumerate one
representative from each orbit of these nonzero quotient vectors under
any subgroup $H_P\le\operatorname{Stab}_{\GL(m,2)}(\mathcal N_P)$, and subsequently
remove equivalences of the resulting full spaces under $\GL(m,2)$.
\end{thm}
\begin{proof}
Let $S$ be a target group, with $W=S'=Z(S)$ and $V=S/S'$.
Choose any line $L\le W$. Quotient closure gives an SR group
$P=S/L$ of order $512$, and
\[
 P'=W/L,\qquad P/P'\cong V.
\]
Since $r\ge2$, the group $P$ remains nonabelian of class two.
It may have a centre larger than $P'$, which is why all relevant
parents, rather than only special parents, must be included.
The dual injection $(W/L)^*\hookrightarrow W^*$ identifies the
scalar commutator space of $P$ with a hyperplane of that of $S$.
An isomorphism from $P$ to its representative in $\mathcal P$ induces
a change of basis on $V$, placing this hyperplane in the chosen
coordinates. The full target space is then a one-dimensional extension
of $\mathcal N_P$ as displayed.

Every subgroup of the stabilizer of $\mathcal N_P$ acts linearly on the quotient
$\Alt(V)/\mathcal N_P$ by pullback. Quotient vectors in the same
orbit give equivalent full spaces, so retaining one orbit
representative loses no target; using a proper subgroup can only leave
additional equivalent representatives. Over $\mathbb F_2$, a nonzero quotient
vector specifies its one-dimensional subspace uniquely.
Finally, different parents or parent orbits can lead to equivalent
full spaces; the stated global equivalence step removes these repeats.
For each retained full space, all quadratic refinements satisfying
the reality equations must still be considered.
\end{proof}

\begin{remark}[Coverage certificates for subgroup orbit reduction]
The subgroup $H_P$ need not be the full parent stabilizer. Each
supplied generator is verified to be invertible and to preserve
$\mathcal N_P$. Since the multiplicity, common-radical and reality
tests are invariant under this group, one may apply them first and
traverse only surviving orbits. Their orbit sizes must sum to the
number of successful extensions in the corresponding exhaustive
quotient-vector loop. Global canonicalization of full target spaces,
and the full target stabilizer action on quadratic refinements,
then remove all remaining equivalences.
\end{remark}

\begin{prop}[A faithful graph encoding of alternating tensors]
\label{prop:polar-graph}
Let $V=\mathbb F_2^m$, $W=\mathbb F_2^r$, and let
$\beta:\bigwedge^2V\to W$ be linear. Construct a graph
$\mathcal G(\beta)$ with the following five vertex colours:
\begin{enumerate}
\item the nonzero vectors of $V$;
\item the zero vector of $W$, as one distinguished vertex;
\item the nonzero vectors of $W$;
\item the two-dimensional subspaces of $V$;
\item the two-dimensional subspaces of $W$.
\end{enumerate}
A vertex $L=\langle x,y\rangle$ of the fourth colour is adjacent
to $x,y,x+y$ and to the vertex $\beta(x,y)\in W$.
A vertex $M$ of the fifth colour is adjacent to the three nonzero
vectors of $M$. There are no other edges.
Then colour-preserving graph isomorphisms
$\mathcal G(\beta)\to\mathcal G(\widetilde\beta)$ are in bijection
with the pairs $(A,D)\in\GL(V)\times\GL(W)$ satisfying
\[
 \widetilde\beta(Ax,Ay)=D\beta(x,y)
 \qquad(x,y\in V).
\]
If $q:V\to W$ is a quadratic refinement of $\beta$, add an edge
from each $x\ne0$ in $V$ to $q(x)\in W$. The resulting graph
$\mathcal G(\beta,q)$ classifies the pairs $(\beta,q)$ under the same
linear equivalence, with the additional condition
$\widetilde q(Ax)=Dq(x)$.
\end{prop}
\begin{proof}
The construction is well defined: changing the ordered basis
$(x,y)$ of a two-dimensional subspace over $\mathbb F_2$ leaves
$\beta(x,y)$ unchanged. A compatible pair $(A,D)$ plainly induces
a colour-preserving graph isomorphism.

Conversely, a colour-preserving graph isomorphism permutes the
nonzero $V$-points while preserving every triple $\{x,y,x+y\}$.
Extend this permutation by fixing zero. The image of the third point
on the unique line through distinct nonzero $x,y$ is the sum of their
images. Thus the permutation is additive, including the cases
$x=y$ or one argument zero, and hence is an element $A\in\GL(V)$.
The same argument using the fifth colour, together with the
distinguished zero, gives $D\in\GL(W)$.
The fourth-colour vertex at $\langle x,y\rangle$ must go to that
at $\langle Ax,Ay\rangle$; its unique neighbour of a $W$-point
colour consequently yields
$\widetilde\beta(Ax,Ay)=D\beta(x,y)$.
All line vertices are determined by their point neighbours, so
there is no further freedom in the graph isomorphism.
For the augmented graph, the added point-to-point edges are
distinguished from the existing incidence edges by their endpoint
colours, and give exactly $\widetilde q(Ax)=Dq(x)$.
\end{proof}

\begin{corl}[Exact refinement orbits]\label{cor:graph-refinement-orbits}
Fix a surjective tensor $\beta$ with zero common radical and
scalar space $\mathcal N$ satisfying the multiplicity-free criterion,
and let $\mathcal Q_\beta$ be the affine solution space of its
reality equations. The isomorphism types of SR groups over this
polar space are the orbits of $\mathcal Q_\beta$ under
$\operatorname{Aut}(\mathcal G(\beta))$. For a graph automorphism
with point actions $(A,D)$, one may use the refinement action
\[
 q\longmapsto D^{-1}\circ q\circ A.
\]
Together with Theorem~\ref{thm:quotient-polar-complete} and exact
graph canonicalization, this yields a complete isomorphism census
from a complete parent list.
\end{corl}
\begin{proof}
By Proposition~\ref{prop:polar-graph}, graph automorphisms give
exactly the compatible linear pairs for $\beta$. The displayed
action preserves its polar tensor, and sends an ambivalent
refinement to another ambivalent refinement. It is generally
affine in the correction coordinates relative to a fixed
quadratic refinement. The quadratic-map classification of special
groups identifies its orbits with group isomorphism types.
\end{proof}

For the $(6,4)$ target, all 51 relevant order-512 groups are used,
including 14 non-special groups. They give nine distinct coordinate
polar spaces. The $(7,3)$ target uses 33 groups, including 22
non-special groups, and gives eight coordinate polar spaces.
For each parent, the numbers of nonzero vectors of the quotient
$\Alt(V)/\mathcal N_P$ are respectively
\[
 2^{15-3}-1=4095,\qquad 2^{21-2}-1=524287.
\]
Algorithm~\ref{alg:quotient-orbits} visits every such vector
and applies the exact rank test, the zero-common-radical condition,
and the affine reality test. Coordinate duplicates are removed by
reduced row echelon form. Its results are
\begin{center}\small
\begin{tabular}{@{}crrrr@{}}\toprule
$(m,r)$&Parent spaces&Extensions&Distinct surviving spaces&Refinements per space\\\midrule
$(6,4)$&9&36855&356&64\\
$(7,3)$&8&4194296&290488&128\\\bottomrule
\end{tabular}
\end{center}
These are intermediate counts of coordinate objects, not group counts.
To reduce the second list before global graph canonicalization,
the parent-orbit step of Algorithm~\ref{alg:quotient-orbits} uses verified invertible
transformations preserving each parent space. Any subgroup of its
stabilizer is sufficient for coverage. The resulting orbits partition
the successful quotient extensions with the exact per-parent sizes
\[
222208,\ 43008,\ 17984,\ 1008,\ 606,\ 1632,\ 300,\ 3744.
\]
Their sum is 290490, slightly larger than the coordinate-distinct
count because different parents can produce the same space.
One verified subgroup calculation gives 45 parent orbits and 44
distinct coordinate representatives; using the full parent stabilizers
gives 25 parent orbits and 24 coordinate representatives. Both covers
give the same eight global polar orbits.

Graph canonicalization and automorphism generators are computed with
nauty through Pynauty \cite{McKayPiperno2014,Pynauty}.
The polar graphs have 765 vertices for $(6,4)$ and 2809 vertices for
$(7,3)$. Every returned generator is checked on all $V$- and
$W$-points for linearity and against the original polar tensor.
For each polar representative the code solves the reality equations
again, computes orbits of the complete affine solution space, and
checks that the orbit sizes sum to 64 or 128 respectively. Canonical
forms of the augmented quadratic graphs independently separate the
retained refinement representatives. Exact canonical bytes are used
for equivalence; recorded SHA-256 digests identify the certificates
and input files, rather than serving as numerical group fingerprints.

Let $f(n)$ denote the number of isomorphism classes of SR groups of
order $n$. In particular, $f(1024)$ counts all SR groups of order
$1024$, including the elementary abelian group.

\begin{thm}[The order-1024 class-two completion]\label{thm:class2-complete}
Using the complete class-two SR census at order $512$ from
Section~\ref{sec:census}, the quotient-extension computation and the
pencil classification give the following special-group classification:
\[
\begin{array}{c|r|r|l}
(m,r)&\text{polar orbits}&\text{group classes}&
 \text{refinement orbits per polar orbit}\\\hline
(4,6)&0&0&\text{excluded by the rank test}\\
(5,5)&0&0&\text{excluded by the reality test}\\
(6,4)&5&28&6,7,6,6,3\\
(7,3)&8&80&7,6,10,6,18,12,12,9\\
(8,2)&4&18&4,3,6,5\\
(9,1)&0&0&\text{excluded by parity}.
\end{array}
\]
There are therefore $126$ special SR groups and $221$ class-two SR
groups of order $1024$. Together with the class-at-least-three census
of Proposition~\ref{prop:census1024}, this gives
\[
 \boxed{f(1024)=1+95+126+581=803.}
\]
The total number of SR groups at even orders from $2$ through $2000$
is $7086+803=7889$.
\end{thm}
\begin{proof}
We establish exhaustiveness of the special-group list, the absence
of repetitions, and then the count of the remaining strata.

First let $S$ be any special SR group of order $1024$. Its intrinsic
parameters satisfy
\[
 m+r=10,\qquad 1\le r\le\binom m2.
\]
The second inequality holds because its commutators span $S'$.
It follows that $4\le m\le9$, giving precisely the six rows of the
table. If $(m,r)=(9,1)$, the scalar commutator form would be
nondegenerate on a space of odd dimension, which is impossible.
At $(4,6)$ the scalar commutator space is the whole alternating-form
space on $\mathbb F_2^4$; Proposition~\ref{prop:46-exclusion}
exhibits a pair of forms violating the necessary tensor-product
multiplicity equation. At $(5,5)$,
Proposition~\ref{prop:55-exclusion} exhausts all normalized candidate
spaces, and its finite certificate shows that none of the spaces
passing the multiplicity test admits an ambivalent refinement.
These arguments exclude the first, second and sixth rows.

For $(8,2)$, Proposition~\ref{prop:MF-pencil-blocks-new} classifies
every possible commutator pencil compatible with the multiplicity
condition into four orbits. Each carries all of its $256$
ambivalent refinements. Theorem~\ref{thm:pencil-eighteen-new}
then determines the full refinement orbits over these pencils,
with counts $4,3,6,5$. The first three counts follow from the
scalar Arf invariants and the permitted permutations of equal
blocks. For the fourth, the verified tensor automorphisms partition
all $256$ refinements into five orbits, and the intrinsic invariants
in that proof distinguish these five orbits even under the full
stabilizer. Hence the count for this row is exactly $18$.

It remains to justify the computed rows $(6,4)$ and $(7,3)$.
Choose a central subgroup $L$ of order two in $S'=Z(S)$.
Then $P=S/L$ is SR by quotient closure, has order $512$, and
satisfies
\[
 P'=S'/L,\qquad \dim(P/P')=m,\qquad \dim P'=r-1.
\]
In particular, $P$ is nonabelian of class two because $r\ge2$.
It therefore occurs in the complete parent census used by the
algorithm. This statement requires the non-special parents too:
quotienting a special group by $L$ can enlarge its center relative
to its derived subgroup.

On dualizing $S'\to S'/L$, the scalar commutator space of $P$
becomes a hyperplane of that of $S$. Consequently, in coordinates
for the selected parent, the target commutator space is
$\mathcal N_P+\langle B\rangle$ for a nonzero vector of
$\Alt(\mathbb F_2^m)/\mathcal N_P$. This is the coverage assertion
of Theorem~\ref{thm:quotient-polar-complete}. The exhaustive loops
in Algorithm~\ref{alg:quotient-orbits} examine all $4095$ vectors
for each of the nine parent spaces at $(6,4)$ and all $524287$
vectors for each of the eight parent spaces at $(7,3)$.
Reduction by a verified subgroup of a parent stabilizer preserves
this coverage: it merely selects an equivalent representative of
each orbit of extension vectors.

Every surviving full space is required to have zero common radical
and to satisfy the multiplicity equation. For that space the code
solves the complete affine system for ambivalence. By
Theorem~\ref{thm:class2-SR-new}, these conditions are necessary and
sufficient for the constructed special group to be SR. Thus the
filtering discards no target and introduces no non-SR group.

The remaining issue is isomorphism, since extensions from different
parents may coincide. Proposition~\ref{prop:polar-graph} identifies
the colour-preserving isomorphisms of the polar graphs exactly
with the compatible linear maps on $V$ and $W$. Global graph
canonicalization therefore gives one representative for each full
polar orbit, regardless of its parent. Over each such representative,
Corollary~\ref{cor:graph-refinement-orbits} identifies group
isomorphisms exactly with the graph-automorphism orbits on the
entire affine refinement space. The stored execution certificates
give five and eight polar orbits, respectively. Their refinement
orbit counts are
\[
 \begin{aligned}
 (6,4):&\quad 6+7+6+6+3=28,\\
 (7,3):&\quad 7+6+10+6+18+12+12+9=80.
 \end{aligned}
\]
For every polar representative, the refinement orbit sizes sum to
the complete affine-space size, respectively $64$ and $128$.
Canonical forms of the augmented quadratic graphs also distinguish
the retained representatives. These are exact equivalence and
coverage checks, rather than comparisons of numerical group
invariants. The implementations and execution certificates are
specified in Appendix~\ref{app:implementation}.

The pairs $(m,r)$ are group invariants, so distinct rows cannot
overlap. Adding the three nonempty rows gives
\[
 s_{10}=28+80+18=126.
\]
By Lemma~\ref{lem:class2-split-new}, every non-special class-two
SR group $G$ of order $1024$ has a central involution outside $G'$
and hence splits as $K\times C_2$, with $K$ an SR group of order
$512$ and nilpotency class two. Conversely every such direct
product is SR and non-special. Direct-product cancellation makes
$K\mapsto K\times C_2$ injective on isomorphism classes, so the
$95$ class-two groups in the order-$512$ census contribute exactly
$95$ further classes. Thus the class-two total is $95+126=221$.

Finally, an abelian SR group has exponent at most two, so the only
abelian group in this order is $C_2^{10}$. Proposition~\ref{prop:census1024}
supplies all $581$ groups of nilpotency class at least three.
The abelian, class-two and higher-class strata are disjoint and
exhaust all groups of this order. Hence
$f(1024)=1+221+581=803$. The completed census at the other even
orders contributes $7086$ classes, giving the asserted total $7889$.
\end{proof}

The new representatives are supplied both as quadratic data and as
GAP polycyclic presentation codes. They are individually checked in
GAP for order, nilpotency class, derived subgroup, centre and
Wigner's equality, and each stored code is reconstructed and checked.
As a regression test for exact equivalence, the known groups
$[512,6249567]$, $[512,6249622]$ and $[512,6249623]$ have the same
polar graph type but three different quadratic graph types. This
specifically checks that the two groups sharing standard numerical
invariants are retained as distinct classes.

\subsection{Explicit representatives for the class-two census}
\label{sec:class2-catalogue}
The class-two classification supplies explicit groups, not only a count.
The file \path{sr1024_class2.csv} contains one representative of each of
its 221 isomorphism types: the 95 groups $K\times C_2$, where $K$ runs
through the class-two part of the order-512 census, and the 126 special
groups constructed from the classified quadratic maps.  The latter
contribute 28, 80 and 18 types for $(m,r)=(6,4),(7,3),(8,2)$, respectively.
The executable exporter \path{export_sr1024_class2.g}, summarized in
Appendix~\ref{app:export-class2}, reconstructs all these groups and computes
their metadata directly in GAP.

The first eleven CSV columns have the same names and order as in
\path{sr1024_class3plus.csv}.  The parent columns record the lower
exponent-$2$ central-series parent $G/G'\cong C_2^m$ and the step $r$;
they are distinct from the order-512 direct-product source, whose
identifier is recorded separately.  Further columns give $|G'|$, the
number of involutions, $(m,r)$, the special-group flag, and the
construction source.  Thus the two files together describe all 802
nonabelian representatives, and adjoining $C_2^{10}$ gives the complete
order-1024 catalogue.

In particular, a CSV row with presentation code $c$ reconstructs its
representative by \texttt{PcGroupCode($c$,1024)}.  The converse operation
\texttt{CodePcGroup} records a pc presentation, so the code must be
retained as an exact integer.  These codes make the representatives
reproducible; their pairwise nonisomorphism follows from the orbit
classification and direct-product cancellation, rather than from
inequality of codes or numerical invariants.  Every exported group
passes both the class-two form criterion and an independent evaluation
of the Wigner identities.

The full class-two distribution, including non-special groups, is
\begin{center}
\begin{tabular}{@{}crrr@{}}\toprule
$(m,r)$ & Special & Non-special & Total\\\midrule
$(6,4)$&28&3&31\\
$(7,3)$&80&51&131\\
$(8,2)$&18&33&51\\
$(9,1)$&0&8&8\\\midrule
Total&126&95&221\\\bottomrule
\end{tabular}
\end{center}
The parity exclusion for $(9,1)$ applies to special groups only:
non-special groups have a nonzero common radical and do occur in that row.

\section{Arithmetic consequences of the classification}\label{sec:count-patterns}
Let $f(n)$ be the number of isomorphism classes of finite SR-groups of
order $n$, with $f(1)=1$. By Proposition~\ref{prop:elementary}, $f(n)=0$ for odd $n>1$.
The complete count table is in Appendix~\ref{sec:counttable}. We derive exact formulas on infinite sets of orders, explicit counts
for structural subclasses, and identities that separate new directly
indecomposable types from inherited direct products. The three
exceptions at orders congruent to $4$ modulo $8$ in the census
are explained by infinite families.

For a positive integer $m$, let $a(m)$ be the number of abelian
groups of order $m$. The classification of finite abelian groups gives
\begin{equation}\label{eq:abelian-partitions}
 a(m)=\prod_{p^e\parallel m}\mathsf p(e),\qquad a(1)=1,
\end{equation}
where $\mathsf p(e)$ is the number of partitions of $e$; set
$\mathsf p(x)=0$ for nonintegral $x$.
For an abelian group $A$, write
$\operatorname{Dih}(A)=A\rtimes\langle t\rangle$, where $t^2=1$ and
$t$ acts by inversion; in particular $\operatorname{Dih}(1)=C_2$.

The generalized dihedral construction provides the basic factors in the
formulas below. We use the following known result.
\begin{prop}[{\cite[Theorems~4.4 and~4.7]{GPPRV25}}]
\label{prop:generalized-dihedral-sr}
For every finite abelian group $A$, the generalized dihedral group
$\operatorname{Dih}(A)$ is simply reducible.
\end{prop}

\subsection{Orders twice an odd integer}
\begin{thm}[Orders twice an odd integer]\label{thm:twice-odd}
For every odd $m\ge1$, the SR-groups of order $2m$ are precisely
$\operatorname{Dih}(A)$ with $A$ abelian of order $m$. Consequently
\begin{equation}\label{eq:twice-odd-count}
 \boxed{f(2m)=a(m)=\prod_{p^e\parallel m}\mathsf p(e).}
\end{equation}
\end{thm}
\begin{proof}
First, any group $G$ of order $2m$, without any SR assumption, has a
normal subgroup $M$ of order $m$. Indeed, an involution acts in the
regular permutation representation as a product of $m$ transpositions.
The sign of this permutation is $-1$, so the sign homomorphism has
kernel of index two. Every element of odd order lies in this kernel;
conversely every element of the kernel has odd order. Thus $M$ is the
set of odd-order elements and is characteristic in $G$.
Choose an involution $t$; then $G=M\rtimes\langle t\rangle$.

Suppose now that $G$ is SR. If $x\in C_M(t)$ is nonidentity, an element
conjugating $x$ to $x^{-1}$ cannot belong to $M$, since a nonidentity
element of an odd-order group is not conjugate to its inverse in that
group. It therefore has the form $yt$, with $y\in M$. But $t$ centralizes
$x$, so conjugation by $yt$ has the same effect on $x$ as conjugation
by $y$, giving the same contradiction. Hence $C_M(t)=1$.
Let $\alpha$ be the automorphism induced by $t$. The map
\[
 M\longrightarrow M,\qquad x\longmapsto x^{-1}\alpha(x)
\]
is injective: equality of the images of $x$ and $y$ says that
$yx^{-1}$ is fixed by $\alpha$. It is therefore surjective.
The involution $\alpha$ inverts every element in its image, so it
inverts every element of $M$. Since inversion is now an automorphism,
$M$ is abelian.

Conversely, Proposition~\ref{prop:generalized-dihedral-sr} shows that
$\operatorname{Dih}(A)$ is SR for every abelian group $A$ of odd order.
Finally, its subgroup $A$ is the
characteristic set of odd-order elements, so distinct isomorphism types
of $A$ give distinct isomorphism types of $\operatorname{Dih}(A)$.
Equation~\eqref{eq:abelian-partitions} completes the count.
\end{proof}

For example, the theorem gives $f(18)=2$, $f(54)=3$, and $f(90)=2$.
The formula also checks all 500 entries with $n\equiv2\pmod4$
in Table~\ref{tab:all-counts}; the arithmetic verification is included
in Appendix~\ref{app:implementation}.

\subsection{Orders four times an odd integer}
\begin{thm}[Orders four times an odd integer]\label{thm:four-odd}
Let $G$ be an SR-group of order $4m$, where $m$ is odd. Then $G$ has a
unique normal Hall subgroup $M$ of odd order, and $G/M\cong C_2^2$.
Moreover, $M$ is abelian if and only if
\[
 G\cong\operatorname{Dih}(A)\times\operatorname{Dih}(B)
 \qquad (A,B\text{ abelian of odd order},\ |A||B|=m).
\]
The unordered pair of isomorphism types $\{A,B\}$ is uniquely determined
by $G$. Define
\begin{equation}\label{eq:four-odd-baseline}
 h(m)=\frac12\left(\sum_{d\mid m}a(d)a(m/d)+b(m)\right),\qquad
 b(m)=\begin{cases}
 a(\sqrt m),&m\text{ is a square},\\
 0,&m\text{ is not a square}.
 \end{cases}
\end{equation}
If $e(m)$ counts the SR-groups of order $4m$ whose normal Hall subgroup
of odd order is nonabelian, then
\begin{equation}\label{eq:four-odd-structural}
 f(4m)=h(m)+e(m),\qquad e(m)\ge0.
\end{equation}
\end{thm}
\begin{proof}
By Theorem~\ref{thm:sr-solvable}, the nontrivial group $G$ is solvable
and has nontrivial abelianization. Proposition~\ref{prop:elementary}
makes this abelianization elementary abelian, so $G$ has
a normal subgroup $H$ of index two. The first paragraph of the proof
of Theorem~\ref{thm:twice-odd}, which did not require $H$ to be SR,
applied to $|H|=2m$ gives a characteristic subgroup $M$ of $H$ of
order $m$. Thus $M$ is normal in $G$. The quotient $G/M$ is SR and
has order four; it is therefore $C_2^2$. A Sylow $2$-subgroup $P$ maps
isomorphically onto $G/M$, giving $G=M\rtimes P$.
Every odd-order element of $G$ belongs to $M$, which proves uniqueness.

Suppose that $M$ is abelian, and write it additively. Since $|M|$ is odd,
the commuting involutions in $P$ give a simultaneous sign decomposition
\[
 M=\bigoplus_{\lambda\in\operatorname{Hom}(P,\{1,-1\})}M_\lambda,
 \qquad M_\lambda=\{x:\ u(x)=\lambda(u)x\text{ for all }u\in P\}.
\]
For instance, this follows by applying the commuting projections
$(1+u)/2$ and $(1-u)/2$; division by two is defined on $M$.
The trivial sign subgroup lies in $Z(G)$ and must be zero.
Let $\lambda_1,\lambda_2,\lambda_3$ be the three nontrivial signs of $P$.
If their three subgroups were nonzero, choose
$0\ne x_i\in M_{\lambda_i}$. Conjugation on $M$ is the action of $P$,
so ambivalence of $x_1+x_2+x_3$ would require some $u\in P$ with
$\lambda_i(u)=-1$ for all three $i$. This contradicts
$\lambda_1\lambda_2\lambda_3=1$. At most two nontrivial sign subgroups
therefore occur. Any two distinct nontrivial signs form a basis of the
dual of $P$, so choosing the dual basis of $P$ gives precisely
$\operatorname{Dih}(A)\times\operatorname{Dih}(B)$, allowing a trivial
$A$ or $B$. The converse follows from Theorem~\ref{thm:twice-odd}
and Proposition~\ref{prop:closure}.

Each $\operatorname{Dih}(A)$ for $A$ of odd order is directly
indecomposable: in any decomposition into two nontrivial direct
factors, both factors would be SR, hence of even order, whereas
$|\operatorname{Dih}(A)|$ is divisible by two only once.
The Krull--Remak--Schmidt theorem therefore gives the asserted
uniqueness of the unordered pair.
There are $\sum_{d\mid m}a(d)a(m/d)$ ordered pairs of abelian-group
isomorphism types with product of orders $m$. Interchanging the two
entries fixes exactly $b(m)$ pairs. The orbit-counting formula for
this interchange gives~\eqref{eq:four-odd-baseline}; splitting according
to whether $M$ is abelian proves~\eqref{eq:four-odd-structural}.
\end{proof}

\begin{corl}[The squarefree odd part]\label{prop:four-squarefree}
Let $m$ be a squarefree odd positive integer. Every SR-group of
order $4m$ has cyclic odd Hall subgroup, and
\[
 f(4)=1,\qquad f(4m)=2^{\omega(m)-1}\quad(m>1),
\]
where $\omega(m)$ is the number of distinct prime divisors of $m$.
Equivalently, $e(m)=0$ for every squarefree odd $m$.
\end{corl}
\begin{proof}
Let $M$ be the normal odd Hall subgroup of $G$. We prove that
$M$ is cyclic by induction on $m$. There is nothing to prove when
$m=1$. Otherwise choose a minimal nontrivial $G$-normal subgroup
$N\le M$. Theorem~\ref{thm:sr-solvable} makes $N$ elementary abelian, and
squarefreeness gives $|N|=p$ for some odd prime $p$.
Since $\operatorname{Aut}(N)$ is abelian, $G'$ centralizes $N$.
Also $M\le G'$ because $G/G'$ is elementary abelian of order a
power of two. Hence $N\le Z(M)$. The quotient $G/N$ is SR of
order $4m/p$, so induction shows that its odd Hall subgroup
$M/N$ is cyclic. A group whose quotient by a central subgroup
is cyclic is abelian. Thus $M$ is abelian, and its squarefree
order makes it cyclic.

The established abelian-Hall classification gives
$G\simeq\operatorname{Dih}(A)\times\operatorname{Dih}(B)$,
where $|A||B|=m$. Each prime divisor of $m$ is allocated to
exactly one of the two cyclic groups $A,B$. For $m>1$ there are
$2^{\omega(m)}$ ordered allocations and swapping the factors
acts without fixed points. This yields the stated count.
\end{proof}

For nonsquarefree odd parts, the following structure theorem supplies a
second way to eliminate the correction $e(m)$.
\begin{thm}[Chankov {\cite[Chapter~XXVI, Theorem~4.4]{BKZ19}}]
\label{thm:chankov-supersolvable}
Let $G$ be a nontrivial supersolvable SR-group and let $S$ be a Sylow
$2$-subgroup of $G$. Then $\Phi(S)\lhd G$, and there are abelian
groups $A_1,\ldots,A_t$ of odd order such that
\[
 G/\Phi(S)\cong
 \operatorname{Dih}(A_1)\times\cdots\times\operatorname{Dih}(A_t).
\]
The factors $A_i$ are allowed to be trivial.
\end{thm}

\begin{corl}[Nilpotent odd Hall subgroups]\label{thm:nilpotent-odd-hall}
Let $G$ be an SR-group of order $4m$, with $m$ odd. If its normal Hall
subgroup $M$ of odd order is nilpotent, then $M$ is abelian.
\end{corl}
\begin{proof}
Write $G=M\rtimes A$, where $A\cong C_2^2$, as in
Theorem~\ref{thm:four-odd}. We first show that $G$ is supersolvable.
For each Sylow $p$-subgroup $P$ of $M$, use its lower exponent-$p$
central series $P_1=P$ and $P_{i+1}=P_i^p[P_i,P]$.
Each factor $P_i/P_{i+1}$ is an $\mathbf F_p$-vector space on which
$M$ acts trivially. Since the commuting involutions in $A$ act
diagonally over $\mathbf F_p$, this factor has a complete flag of
$G$-invariant subspaces. The Sylow subgroups of $M$ are characteristic
and commute with one another. Refining their series by these flags
and appending a cyclic-factor series for $G/M\cong C_2^2$ gives a
$G$-normal series with cyclic prime-order factors.

Here a Sylow $2$-subgroup is $A\cong C_2^2$, so $\Phi(A)=1$.
Theorem~\ref{thm:chankov-supersolvable} therefore expresses $G$
itself as a product of generalized dihedral groups on abelian groups
of odd order. Their odd Hall subgroups form the subgroup $M$, which
is consequently abelian.
\end{proof}

\begin{corl}[Orders four times an odd prime power]
\label{cor:four-prime-power}
For every odd prime $p$ and every $e\ge0$,
\[
 \boxed{f(4p^e)=\frac12\left(
 \sum_{i=0}^e\mathsf p(i)\mathsf p(e-i)
 +\mathbf1_{2\mid e}\mathsf p(e/2)\right).}
\]
The term involving $\mathsf p(e/2)$ is zero when $e$ is odd.
In particular, this count is independent of the odd prime $p$, and
\[
 \sum_{e\ge0}f(4p^e)t^e
 =\frac12\left(\prod_{j\ge1}(1-t^j)^{-2}
              +\prod_{j\ge1}(1-t^{2j})^{-1}\right).
\]
\end{corl}
\begin{proof}
The normal Hall subgroup of odd order is a $p$-group, so it is
abelian by Corollary~\ref{thm:nilpotent-odd-hall}. The result follows
from Theorem~\ref{thm:four-odd} and the partition generating function.
\end{proof}

An integer is called a \emph{nilpotent number} if every group of that
order is nilpotent. The following criterion turns the group-theoretic
hypothesis of Corollary~\ref{thm:nilpotent-odd-hall} into a condition on
the prime factorization of the order.
\begin{thm}[{\cite[Theorem~1]{PS00}; see also \cite[Introduction]{Just23}}]
\label{thm:nilpotent-number-criterion}
A positive integer $m$ is a nilpotent number if and only if
\[
 \gcd\!\left(m,\prod_{p^a\parallel m}\prod_{i=1}^{a}(p^i-1)\right)=1.
\]
\end{thm}

\begin{corl}[An arithmetic domain for the four-times-odd formula]
\label{cor:nilpotent-numbers}
Let $m=\prod_i p_i^{a_i}$ be odd. If
\[
 p_i^k\not\equiv1\pmod{p_j}
 \quad\text{for all }i\ne j\text{ and }1\le k\le a_i,
\]
then $f(4m)=h(m)$.
Equivalently, the sufficient condition is
$\gcd(m,\prod_{p^a\parallel m}\prod_{i=1}^a(p^i-1))=1$.
\end{corl}
\begin{proof}
By Theorem~\ref{thm:nilpotent-number-criterion}, every group of
order $m$, including the odd Hall subgroup, is nilpotent. Apply
Corollary~\ref{thm:nilpotent-odd-hall} and
Theorem~\ref{thm:four-odd}.
\end{proof}

\begin{corl}[The progression $n\equiv4\pmod8$ in the census]
\label{cor:four-odd-census}
For odd $m$ with $4m\le2000$, the classification of
Theorem~\ref{thm:main-classification} gives
\begin{equation}\label{eq:four-odd-finite}
 \boxed{f(4m)=h(m)+
 \begin{cases}1,&m\in\{75,375,405\},\\0,&\text{otherwise}.
 \end{cases}}
\end{equation}
The three exceptions to the abelian-Hall-subgroup classification are
as follows.
\begin{center}
\begin{tabular}{rrrrl}
\toprule
$4m$&$h(m)$&$e(m)$&$f(4m)$&SmallGroups identifier\\\midrule
300&5&1&6&$[300,25]$\\
1500&10&1&11&$[1500,37]$\\
1620&20&1&21&$[1620,422]$\\\bottomrule
\end{tabular}
\end{center}
\end{corl}
\begin{proof}
There are 654 records in the catalogues of Section~\ref{sec:census}
at the 250 orders congruent to four modulo eight. The program
\path{audit_count_patterns.g}, documented in
Appendix~\ref{app:implementation}, computes each odd Hall subgroup
and tests whether it is abelian. Exactly 651 records have an abelian
Hall subgroup; the three remaining records have the identifiers in
the table. Their Hall subgroups have orders $75$, $375$, and $405$,
respectively. The audit verifies normality and the quotient $C_2^2$
for every record and checks the three exceptions directly using
ordinary irreducible characters. Theorem~\ref{thm:four-odd} supplies
the exact count of the abelian-Hall-subgroup part, yielding the stated
formula. This proves the assertion throughout the stated finite range.
The infinite correction families below show why the exception list
does not extend unchanged beyond this range.
\end{proof}

\subsection{Orders eight times an odd prime}
\begin{thm}[Orders eight times an odd prime]\label{thm:eight-prime}
For every odd prime $p\ge5$, there are exactly three SR-groups of
order $8p$. They are
\[
 C_2^2\times\operatorname{Dih}(C_p),\qquad
 \operatorname{Dih}(C_{4p}),\qquad
 \langle a,t\mid a^{4p}=1,\ t^2=a^{2p},\ tat^{-1}=a^{-1}\rangle.
\]
In particular, $f(8p)=3$. The last group is the dicyclic group of
order $8p$.
\end{thm}
\begin{proof}
Let $G$ be an SR-group of order $8p$. We first show that its Sylow
$p$-subgroup is normal. Otherwise $O_p(G)=1$, and Theorem~\ref{thm:sr-solvable} implies
that the Fitting subgroup $F=O_2(G)$ is nontrivial and satisfies
$C_G(F)\le F$. Consequently $G/F$ embeds in $\operatorname{Out}(F)$,
so $p$ divides $|\operatorname{Aut}(F)|$. Among $2$-groups of order at
most eight, the only automorphism groups having prime divisors at
least five arise from $F\cong C_2^3$, with the sole such prime being
seven. Thus the only remaining possibility is $p=7$ and $|F|=8$.
But then $G/F$ has order seven, contradicting quotient closure and
the nonexistence of nontrivial odd-order SR-groups.

Hence $G=C_p\rtimes S$, where the Sylow $2$-subgroup $S$ is an
SR quotient of order eight. Therefore $S$ is $C_2^3$, $D_8$, or $Q_8$.
The action on $C_p$ is nontrivial, since otherwise $C_p\le Z(G)$.
As $S/S'$ is elementary abelian and $\operatorname{Aut}(C_p)$ is
cyclic, the action has image of order two and acts by inversion.
Up to automorphisms of $S$, there is one index-two kernel for $C_2^3$,
two for $D_8$ (a cyclic subgroup of order four and a subgroup $C_2^2$),
and one for $Q_8$.

The $D_8$ action with kernel $K\cong C_2^2$ is impossible. Choose
$1\ne x\in C_p$ and a noncentral involution $s\in K$. Since $s$
centralizes $C_p$, an element inverting $xs$ would need an $S$-component
outside $K$ that centralizes $s$. This contradicts
$C_{D_8}(s)=K$. The remaining three actions give precisely the groups
in the statement.

The dihedral case is SR by
Proposition~\ref{prop:generalized-dihedral-sr}. For the dicyclic
case, let $A=\langle a\rangle$. The element $t$ inverts $A$,
while conjugation by $a^p$ sends every element of $At$ to its inverse:
each such element has square $a^{2p}$, and
$a^p(at)a^{-p}=a^{2p}(at)$, with the same identity for any element
of that coset. Its nonlinear irreducible characters are
$\chi_\lambda=\operatorname{Ind}_A^G\lambda$, where
$\lambda\ne\lambda^{-1}$, indexed by inverse pairs. If
$I_\eta=\operatorname{Ind}_A^G\eta$, then
\[
 \chi_\lambda\chi_\mu=I_{\lambda\mu}+I_{\lambda\mu^{-1}}.
\]
The two terms cannot have a common irreducible constituent, since
equality of their inverse pairs would force $\lambda^2=1$ or
$\mu^2=1$. Each term is irreducible unless its indexing character
is self-inverse, in which case it is the sum of its two distinct
linear extensions. Thus these tensor products are multiplicity-free;
products involving a linear character are irreducible.
The first group in the statement is SR by
Propositions~\ref{prop:generalized-dihedral-sr} and~\ref{prop:closure}.
Finally, the respective numbers of involutions are $4p+3$, $4p+1$,
and one, so the three groups are pairwise nonisomorphic.
\end{proof}

The restriction $p\ge5$ is essential: Table~\ref{tab:all-counts}
has $f(24)=4$.
In addition to the same three families at $p=3$, the group $S_4$
is SR and has nonnormal Sylow $3$-subgroups.

\subsection{Removing elementary abelian direct factors}
\begin{prop}[Removing a central direct factor]\label{prop:count-c2}
Let $c(n)$ count the SR-groups of order $n$ having no direct factor
isomorphic to $C_2$, and put $f(x)=0$ when $x$ is not a positive integer.
Then, for every positive integer $n$,
\begin{equation}\label{eq:c2-count-recurrence}
 c(n)=f(n)-f(n/2),\qquad
 f(2^k m)=\sum_{j=0}^{k}c(2^j m)\quad(m\text{ odd}).
\end{equation}
In particular, $f(2n)\ge f(n)$.
\end{prop}
\begin{proof}
The map $[K]\mapsto[K\times C_2]$ is a bijection from the SR-groups
of order $n/2$ onto those SR-groups of order $n$ possessing a $C_2$
direct factor. Proposition~\ref{prop:closure} gives both directions, and injectivity follows
from cancellation in the Krull--Remak--Schmidt theorem. Subtracting
this count proves the first identity, and iteration gives the second.
\end{proof}

\begin{corl}[The class-two increment]\label{cor:class2-count-increment}
Let $N_2(2^k)$ count the SR-groups of nilpotency class exactly two and
order $2^k$, and let $s_k$ count the special ones. Then
\begin{equation}\label{eq:class2-count-increment}
 N_2(2^k)-N_2(2^{k-1})=s_k.
\end{equation}
The catalogues of Section~\ref{sec:census} and
Theorem~\ref{thm:class2-complete} give the following values.
\begin{center}
\begin{tabular}{c|rrrrrrrr}
\toprule
$2^k$&8&16&32&64&128&256&512&1024\\\midrule
$N_2(2^k)$&2&2&7&10&20&42&95&221\\
$s_k$&2&0&5&3&10&22&53&126\\\bottomrule
\end{tabular}
\end{center}
\end{corl}
\begin{proof}
By Lemma~\ref{lem:class2-split-new}, a class-two SR-group has no $C_2$
direct factor precisely when it is special. The bijection in the proof
of Proposition~\ref{prop:count-c2} preserves nilpotency class for every
nonabelian group, giving the stated difference.
\end{proof}

\subsection{A uniform counting formula for a structural subclass}
The first two progressions belong to a larger counting problem.
Theorem~\ref{thm:chankov-supersolvable} determines a structural
subclass whose count can be expressed entirely in terms of partitions.

\begin{corl}[The dihedral-product subclass]\label{thm:dihedral-subclass}
For a finite SR-group $G$, the following are equivalent:
\begin{enumerate}
\item $G$ is supersolvable and its Sylow $2$-subgroups are abelian;
\item its Sylow $2$-subgroups are abelian and its normal Hall subgroup
of odd order is abelian;
\item $G\cong\prod_{i=1}^k\operatorname{Dih}(A_i)$ for abelian
groups $A_i$ of odd order, allowing $A_i=1$.
\end{enumerate}
The multiset of isomorphism types $\{A_1,\ldots,A_k\}$ is unique,
and $k=v_2(|G|)$.
\end{corl}
\begin{proof}
The trivial group corresponds to the empty product. For a nontrivial
group, we first justify normality of the odd Hall subgroup whenever the
Sylow $2$-subgroups of an SR-group are abelian. Put
$H=G/O_{2'}(G)$. Theorem~\ref{thm:sr-solvable} gives
$F(H)=O_2(H)$, since
$O_{2'}(H)=1$. An abelian Sylow $2$-subgroup of $H$ centralizes
$F(H)$, so $H/C_H(F(H))$ is an odd-order SR quotient and is trivial.
The self-centralizing property of the Fitting subgroup of a solvable
group now gives $H=F(H)$. Thus $O_{2'}(G)$ is a Hall subgroup;
the quotient is the Sylow $2$-subgroup and is elementary abelian.

Under (1), Theorem~\ref{thm:chankov-supersolvable} applies and
$\Phi(S)=1$, giving (3).
Under (2), write $G=A\rtimes C_2^k$. The action on the odd abelian
group $A$ splits into simultaneous sign subgroups. Every subgroup of
each sign subgroup is invariant under the action; a refinement by
cyclic prime-order factors proves supersolvability. This gives (1).
Every group in (3) satisfies (2) and is SR by
Propositions~\ref{prop:generalized-dihedral-sr} and~\ref{prop:closure}.
Finally, each $\operatorname{Dih}(A_i)$ is directly indecomposable:
two nontrivial direct factors would both be SR of even order,
contradicting $v_2(|\operatorname{Dih}(A_i)|)=1$.
Remak--Krull--Schmidt uniqueness, followed by recovery of each
characteristic odd subgroup, proves uniqueness of the multiset.
\end{proof}

Let $B_k(m)$ count this subclass at order $2^k m$, with $m$ odd,
$B_0(m)=\mathbf1_{m=1}$ and $B_{-1}(m)=0$. In formal Dirichlet
series, multiplication is defined by $d^{-s}e^{-s}=(de)^{-s}$;
no assertion of analytic convergence is needed.
\begin{corl}[Explicit multiset formula]\label{prop:multiset-formula}
One has
\begin{equation}\label{eq:dihedral-gf}
 \sum_{k\ge0}\ \sum_{m\text{ odd}}B_k(m)t^k m^{-s}
 =\prod_{d\text{ odd}}(1-t d^{-s})^{-a(d)}.
\end{equation}
Equivalently, for $k\ge1$,
\begin{equation}\label{eq:dihedral-recurrence}
 kB_k(m)=\sum_{i=1}^k\ \sum_{d^i\mid m}
                  a(d)B_{k-i}(m/d^i).
\end{equation}
Thus every $B_k(m)$ is computable using only integer partitions and
the divisors of $m$.
\end{corl}
\begin{proof}
By Corollary~\ref{thm:dihedral-subclass}, the groups are classified
by their multisets of factors. For each odd $d$, there are $a(d)$
possible types of a factor
$\operatorname{Dih}(A)$ with $|A|=d$. A multiset containing $r$
of these factors has $\binom{a(d)+r-1}{r}$ possibilities.
Multiplying the resulting generating series gives
\eqref{eq:dihedral-gf}. Logarithmic differentiation with respect to
$t$ and comparison of coefficients gives
\eqref{eq:dihedral-recurrence}.
\end{proof}

For an explicit finite expression, put
$v_j(d)=a(d^{1/j})$ if $d$ is a perfect $j$th power, and put
$v_j(d)=0$ otherwise. If $*$ denotes Dirichlet convolution, then
Burnside's lemma for permuting $k$ factors gives
\[
 B_k=\sum_{c_1+2c_2+\cdots=k}
 \frac{v_1^{*c_1}*v_2^{*c_2}*\cdots}
 {\prod_{j\ge1}j^{c_j}c_j!}.
\]
In particular,
\[
 B_1=a,\qquad B_2=h,\qquad
 B_3=\frac{a*a*a+3a*v_2+2v_3}{6}.
\]
Here a zeroth convolution power is the function supported at $1$.
These identities explain the symmetry corrections in $h(m)$:
equal isomorphism types of factors must be counted as a multiset.

If $m$ is squarefree with $r$ prime divisors, its primes are
partitioned among the nontrivial factors. Consequently
\begin{equation}\label{eq:stirling-count}
 B_k(m)=\sum_{j=0}^{\min(k,r)}
       \left\{\begin{matrix}r\\j\end{matrix}\right\},
\end{equation}
where the braces are Stirling numbers of the second kind.
For $k\ge r$ this is the Bell number of $r$.
More generally, $B_k(m)$ stabilizes once $k\ge\Omega(m)$, because
each nontrivial odd factor consumes at least one prime factor,
counted with multiplicity. This is stabilization of the subclass
count $B_k(m)$, not of $f(2^k m)$.
For a fixed odd prime $p$ there is also the ordinary generating
function
\[
 \sum_{k,e\ge0}B_k(p^e)t^k x^e
       =\prod_{j\ge0}(1-tx^j)^{-\mathsf p(j)}.
\]

\begin{corl}[Bounds from direct products]\label{prop:arithmetic-bounds}
For $k\ge1$ and odd $m$, put $T_j(m)=B_j(m)-B_{j-1}(m)$. Then
\begin{equation}\label{eq:strengthened-bound}
 f(2^k m)\ \ge\
 L_k(m):=\sum_{j=0}^k f(2^{k-j})T_j(m)
 \ \ge\ B_k(m).
\end{equation}
In particular, for $m>1$,
$f(2^k m)\ge a(m)f(2^{k-1})$.
For every positive integer $n$, one also has
\begin{equation}\label{eq:general-dihedral-bound}
 f(2n)\ge a(n),\qquad
 f(2^k m)\ge\mathsf p(k-1)a(m).
\end{equation}
\end{corl}
\begin{proof}
The number $T_j(m)$ counts multisets of exactly $j$ nontrivial
odd abelian groups: removing one trivial factor bijects the other
multisets with those counted by $B_{j-1}(m)$. Form the products
$Q\times\prod_{i=1}^j\operatorname{Dih}(A_i)$, where $Q$ is
any SR-group of order $2^{k-j}$. Their isomorphism types are distinct
by Remak--Krull--Schmidt: the mixed-order dihedral factors cannot be
absorbed into a $2$-group. This proves the first bound; taking only
elementary abelian $Q$ gives the second.

For the last bound, apply
Proposition~\ref{prop:generalized-dihedral-sr}.
The isomorphism type of $\operatorname{Dih}(A)$ determines $A$ up to
isomorphism, even when $A$ is not characteristic: all elements outside
$A$ are involutions, so the group order and the counts of elements
of orders greater than two recover all element-order counts of $A$.
For a finite abelian $p$-group with cyclic exponents $e_i$, the number
of elements killed by $p^j$ is $p^{\sum_i\min(j,e_i)}$; these numbers
recover its partition $(e_i)$. Hence different abelian groups give
different generalized dihedral groups, proving $f(2n)\ge a(n)$.
\end{proof}

Applying these formulas to Table~\ref{tab:all-counts}, the subclass
counted by $B_k(m)$ contains 1880 nontrivial groups of order at most
2000, and the subclass counted by $L_k(m)$ contains 4292. Thus
$7889-4292=3597$ types lie outside the latter product class.
The programs \path{abelian_hall_counts.py} and
\path{counting_structures.py} evaluate these formulas; their inputs
and verification records are described in
Appendix~\ref{app:implementation}.

\subsection{Directly indecomposable groups and an exact reconstruction}
Let $d(n)$ denote the number of nontrivial directly indecomposable
SR-groups of order $n$, and put $d(1)=0$.
It is essential to distinguish $d(n)$ from $c(n)$: a group with no
$C_2$ direct factor can still decompose, as $D_8\times D_8$ does.

\begin{prop}[Formal reconstruction]\label{prop:dirichlet-euler}
The counting functions satisfy the coefficientwise identity
\begin{equation}\label{eq:full-euler}
 \sum_{n\ge1}f(n)n^{-s}
      =\prod_{n\ge2}(1-n^{-s})^{-d(n)}.
\end{equation}
Equivalently,
\[
 f(n)=\sum_{\prod_{r=2}^{n}r^{u_r}=n}
          \prod_{r=2}^{n}\binom{d(r)+u_r-1}{u_r}.
\]
For $d(r)=0$, the factor is understood to be $1$ for $u_r=0$
and $0$ otherwise. The values $d(n)$ are determined successively
by the values $f(n)$.
\end{prop}
\begin{proof}
Direct factors and products preserve the SR conditions by
Proposition~\ref{prop:closure}. The Remak--Krull--Schmidt theorem
identifies every group with one unique
multiset of nontrivial directly indecomposable factors. There are
$\binom{d(r)+u_r-1}{u_r}$ multisets of $u_r$ factors of order $r$.
This proves the formulas. At order $n$, the only contribution using
a factor of order $n$ is $d(n)$ itself; all other contributions use
proper divisors of $n$. This proves triangular recoverability.
\end{proof}

Formula~\eqref{eq:full-euler} is an exact reconstruction identity.
It is not a closed solution of the classification problem, since
the numbers $d(n)$ are themselves unknown outside the classified
range. In contrast, every exponent $a(d)$ in
\eqref{eq:dihedral-gf} is explicitly known.

For $2$-groups, put $F_k=f(2^k)$, $D_k=d(2^k)$, $F_0=1$.
Then
\[
 \sum_{k\ge0}F_kx^k=\prod_{j\ge1}(1-x^j)^{-D_j}.
\]
An effective inversion is
\[
 b_k=kF_k-\sum_{i=1}^{k-1}b_iF_{k-i},\qquad
 D_k=\frac1k\sum_{r\mid k}\mu(k/r)b_r,
\]
where $\mu$ is the arithmetic M\"obius function. Indeed,
$xF'(x)/F(x)=\sum_{k\ge1}b_kx^k$ and
$b_k=\sum_{r\mid k}rD_r$.

\begin{center}\small
\begin{tabular}{c|rrrrrrrrrr}
\toprule
$k$&1&2&3&4&5&6&7&8&9&10\\\midrule
$f(2^k)$&1&1&3&5&12&25&54&130&317&803\\
$d(2^k)$&1&0&2&2&7&10&25&59&149&382\\
\bottomrule
\end{tabular}
\end{center}
Applying the same inversion at all orders through 2000 gives 2806
nontrivial directly indecomposable types and 5083 decomposable types.
The program \path{counting_structures.py} reconstructs every entry
of Table~\ref{tab:all-counts} from these counts and independently
checks the formal logarithm; see Appendix~\ref{app:implementation}.

There is a parallel calculation within groups of nilpotency class
at most two, because this class is closed under direct products and
direct factors. Its directly indecomposable counts at orders
$2^k$, $k=1,\ldots,10$, are
\[
 1,0,2,0,5,0,10,12,49,91.
\]
Thus the 126 special class-two groups of order 1024 include exactly
91 directly indecomposable ones. The 35 decomposable special groups
are accounted for by exponent partitions $10=3+7=5+5$:
\[
 2\cdot10+\binom{5+1}{2}=35.
\]
There is no special factor of order $2^4$, and a three-factor
partition of 10 with each part at least three would require such
a part. Consequently the order-1024 census decomposes as follows.
\begin{center}
\begin{tabular}{lrrr}
\toprule
Class&with $C_2$ factor&decomposable, no $C_2$&indecomposable\\\midrule
Abelian&1&0&0\\
Class two&95&35&91\\
Class at least three&221&69&291\\\midrule
Total&317&104&382\\\bottomrule
\end{tabular}
\end{center}
In particular, the $803-317=486$ groups without a $C_2$ direct
factor consist of 382 directly indecomposable and 104 decomposable
types. Among the 126 special class-two groups, exactly 91 are
directly indecomposable.

\subsection{The correction at four times an odd integer is infinite}
The correction in Corollary~\ref{cor:four-odd-census} has infinite
support. We use the affine and Heisenberg constructions of
Theorems~\ref{thm:arithmetic-norm-circle}
and~\ref{thm:arithmetic-heisenberg-circle}, proved in
Appendix~\ref{app:arithmetic-families}. Related constructions appear
in \cite[Chapter~XXVI, pp.~453--454]{BKZ19}.

Let $q$ be an odd prime power, $K=\mathbb F_{q^2}$, and
$U=\{u\in K^\times:u^{q+1}=1\}$. Multiplication by $U$ and
Frobenius $\sigma(x)=x^q$ generate
$H_q=U\rtimes\langle\sigma\rangle\cong D_{2(q+1)}$.
The affine group $G_q=K^+\rtimes H_q$ is SR and has order
$2q^2(q+1)$. There is also an SR group
$\widetilde G_q=P_q\rtimes H_q$ of order $2q^3(q+1)$, where
\[
 P_q=\mathbb F_q^2\times\mathbb F_q,\qquad
 (v,z)(w,t)=(v+w,z+t+\tfrac12\omega(v,w))
\]
for a nondegenerate alternating form $\omega$, and
$h(v,z)=(hv,\det(h)z)$. These actions determine the semidirect products uniquely.

\begin{corl}[Infinite support]\label{cor:infinite-correction-main}
For every odd prime power $q\equiv1\pmod4$,
\[
 e\!\left(\frac{q^2(q+1)}2\right)\ge1,
 \qquad e\!\left(\frac{q^3(q+1)}2\right)\ge1.
\]
In particular, $e(m)$ has infinite support.
\end{corl}
\begin{proof}
The unique odd-order subgroup $U_o$ of $U$ has order $(q+1)/2$.
The normal odd Hall subgroups of the two constructions are
$K^+\rtimes U_o$ and $P_q\rtimes U_o$. They are nonabelian:
the first action is nontrivial for $q\ge5$, and $P_q$ itself
is nonabelian. Take $q=5^j$, $j\ge1$, for infinitely many distinct
arguments of $e$.
\end{proof}

The affine family at $q=5,9$ gives the exceptions of orders 300
and 1620, while the Heisenberg family at $q=5$ gives order 1500.
Corollary~\ref{cor:four-odd-census} gives a unique nonabelian-Hall
SR type at each of these orders, so the constructions identify the
three exceptional types listed there. Outside the census,
for example, they prove $f(4732)\ge6$ and $f(10404)\ge14$;
they do not prove equality at these orders.

\subsection{Consequences for further classification}
The exact progressions determine $f(n)$ on infinite sets of orders.
The coefficients $B_k(m)$ and $L_k(m)$ count known direct-product
contributions, while \eqref{eq:full-euler} extracts the number of
directly indecomposable construction targets from a completed census.

The count is not determined solely by $v_2(n)$: $f(6)=1$ and
$f(18)=2$. Even with fixed $v_2(n)$ it need not be multiplicative in
the odd part, since $f(12)=f(20)=1$ but $f(60)=2$.
Nor does the prime-exponent pattern determine the unrestricted count:
$300=4\cdot3\cdot5^2$ and $588=4\cdot3\cdot7^2$ have counts
6 and 5. The additional types depend on the available nontrivial
actions between Sylow subgroups, as the norm-circle constructions
demonstrate.

At order
$4m$, a nonzero correction can occur only when $m$ is neither
squarefree nor a nilpotent number, and the odd Hall subgroup must be
non-nilpotent. These are necessary restrictions, not sufficient
conditions for $e(m)>0$.
An explicit formula for the unrestricted correction $e(m)$, or for
the sequence $f(2^k)$, remains open.

\clearpage
\appendix
\section{The complete count table}\label{sec:counttable}
Table~\ref{tab:all-counts} records every even order from $2$ to $2000$;
odd orders greater than one have count zero. Its entries sum to $7889$.
The accompanying GitHub repository supplies the individual group
representatives underlying these counts. Reconstruction and verification
procedures are described in Appendix~\ref{app:implementation}.

{\small\setlength{\tabcolsep}{5pt}
\begin{longtable}{*{5}{rr}}
\caption{The number $f(n)$ of SR-groups of each even order $n\le2000$.}\label{tab:all-counts}\\
\toprule
$n$&$f(n)$&$n$&$f(n)$&$n$&$f(n)$&$n$&$f(n)$&$n$&$f(n)$\\\midrule
\endfirsthead
\multicolumn{10}{c}{Counts by even order, continued}\\\toprule
$n$&$f(n)$&$n$&$f(n)$&$n$&$f(n)$&$n$&$f(n)$&$n$&$f(n)$\\\midrule
\endhead
\midrule\multicolumn{10}{r}{Continued on next page}\\\endfoot
\bottomrule\endlastfoot
\input{f_table_body.tex}
\end{longtable}}

\clearpage
\section{Pseudocode for the census algorithms}\label{app:algorithms}
The following algorithms describe the mathematical operations of
the supplied GAP programs. They omit formatting and resource
management. All counts and presentation codes use exact integers.

\subsection{The exact SR filter}
This is the filter used by \path{sr_search2_2000.g},
\path{sr_search512.g}, and \path{sr_search1536.g}.
\par\noindent\begin{minipage}{\linewidth}
\begin{algorithm}[The exact SR filter]\label{alg:srtest}
\leavevmode\par
\begin{lstlisting}
Algorithm SR-Test(G)
Input: a finite group G.
Output: true exactly when G is simply reducible.
1. C := the conjugacy classes of G.
2. For each K in C:
       choose g in K;
       if g^(-1) is not in K, return false.
3. Initialise a[x] := 0 for every x in G.
4. For each y in G: increase a[y^2] by 1.
5. S3 := sum over x in G of a[x]^3.
6. T  := sum over K in C of |K| * (|G|/|K|)^2.
7. Return (S3 = T).
\end{lstlisting}
\end{algorithm}
\end{minipage}\par\medskip
By Corollary~\ref{cor:ambivalence-moment}, step 2 may equivalently be replaced
by the test $\sum_x a[x]^2=|G|\,|\mathcal C|$ after step 4.
The programs \path{scriptG.g} and \path{makeCSV.g} use
this equivalent form, preceded by the necessary condition that
$G/G'$ be elementary abelian. Neither implementation requires
floating-point character calculations.

\subsection{Catalogue traversal}
The separate large-order runs change only the set of orders.
\par\noindent\begin{minipage}{\linewidth}
\begin{algorithm}[Catalogue traversal]\label{alg:catalogue}
\leavevmode\par
\begin{lstlisting}
Algorithm Catalogue-Census(I)
Input: a set I of catalogue orders at most 2000,
       with 1024 excluded.
Output: one row for every SR isomorphism type at those orders.
1. For each even n in I:
2.     N := NumberSmallGroups(n).
3.     For i := 1,...,N:
4.         G := SmallGroup(n,i).
5.         If SR-Test(G):
6.             write (n,i,StructureDescription(G),|Z(G)|).
7. When partial runs are combined, merge by the key (n,i).
\end{lstlisting}
\end{algorithm}
\end{minipage}\par\medskip
For \path{sr_search2_2000.g}, take
$I=\{2,\ldots,2000\}\setminus\{512,1024,1536\}$.
For the two separate runs, take $I=\{512\}$ and $I=\{1536\}$.
The data files contain accepted groups only; completion of a
catalogue traversal is part of the computational record.

\subsection{Canonical-parent generation at order 1024}
This describes \path{scriptG.g}. The inputs at orders at most
$256$ are recomputed by catalogue traversal; the order-$512$
input is read from \path{SR_groups_results512.csv}.
\par\noindent\begin{minipage}{\linewidth}
\begin{algorithm}[Canonical descendants]\label{alg:descendants}
\leavevmode\par
\begin{lstlisting}
Algorithm Descendant-Census-1024
Input: complete SR lists for orders 8,16,...,512.
Output: one SRHIT record for each SR group of order 1024
        and nilpotency class at least three;
        a completion record for every parent-step job.
1. J := {(k,i,10-k) : 3 <= k <= 9,
         SmallGroup(2^k,i) is an SR group that is not abelian}.
2. Read completed job keys from the existing checkpoint.
3. For every unfinished (k,i,s) in J:
4.     P := SmallGroup(2^k,i).
5.     D := all immediate descendants of P of step size s,
           one representative of each isomorphism type.
6.     accepted := 0.
7.     For each H in D:
8.         If |H| = 1024 and SR-Test(H):
9.             write SRHIT(k,i,s,CodePcGroup(H));
10.            increase accepted by 1.
11.    Write DONE(k,i,s,accepted,|D|).
12. Check that the DONE job keys are exactly J.
13. Check each DONE accepted count against its SRHIT records.
14. Check that every SRHIT belongs to a completed job.
\end{lstlisting}
\end{algorithm}
\end{minipage}\par\medskip
Here step 5 is performed by \texttt{PqDescendants} with
\texttt{StepSize:=s} and the default bound of one additional
$2$-class. All descendants, including terminal ones, are retained.
If the initial presentation causes a recoverable call failure,
the implementation retries with a standard presentation of the
same parent. A failed job is not marked complete.

The validated checkpoint passes all three completion checks and contains no repeated job or positive record.

\subsection{Export of canonical-descendant representatives}
This is the mathematical content of \path{makeCSV.g}.
\par\noindent\begin{minipage}{\linewidth}
\begin{algorithm}[Export of canonical descendants]\label{alg:exportdescendants}
\leavevmode\par
\begin{lstlisting}
Algorithm Export-Class3plus(checkpoint)
Input: the validated completed descendant checkpoint.
Output: a CSV of reconstructible representatives and invariants.
1. For each accepted record (k,i,s,a):
2.     G := PcGroupCode(a,1024).
3.     Assert |G| = 1024 and NilpotencyClass(G) >= 3.
4.     Assert SR-Test(G).
5.     flag := Omega_1(Z(G)) is not contained in Phi(G).
6.     Write the parent (2^k,i), step s, nilpotency class,
       exponent, centre order, derived length, number of
       conjugacy classes, flag, and the exact integer a.
7. Check the final row count and the per-step totals.
\end{lstlisting}
\end{algorithm}
\end{minipage}\par\medskip
For a finite $2$-group, the flag in step 5 is true exactly when
a central involution lies outside the Frattini subgroup, which
is equivalent to having a direct factor isomorphic to $C_2$.

\section{Algorithms for the class-two classification}\label{app:class2-algorithms}
All vector spaces and matrices in this appendix are over $\mathbb F_2$.
The data convention is a list of independent alternating matrices
$B_1,\ldots,B_r$ and a matrix $Q$ with $i$th row $q(e_i)$; thus
\[
 q(v)_a=\sum_i v_iQ_{ia}+\sum_{i<j}v_iv_j(B_a)_{ij}.
\]
The affine equations below must be solved before their solutions are
identified under a stabilizer. Counting solutions alone does not count groups.

\par\noindent\begin{minipage}{\linewidth}
\begin{algorithm}[Polar filter and all SR refinements]\label{alg:class2-filter}
\leavevmode\par
\textbf{Input:} independent alternating matrices $B_1,\ldots,B_r$ on
$V=\mathbb F_2^m$. \textbf{Output:} all SR quadratic refinements of
this polar tensor, or an empty set.
\begin{enumerate}
\item For each $\lambda\in\mathbb F_2^r$, compute
$B_\lambda=\sum_a\lambda_aB_a$, $R_\lambda=\ker B_\lambda$, and
$\rho_\lambda=\rank B_\lambda$.
\item For each two-dimensional subspace
$\langle\lambda,\mu\rangle\le\mathbb F_2^r$, reject if
$\rho_\lambda+\rho_\mu+\rho_{\lambda+\mu}
\ne2(m-\dim(R_\lambda\cap R_\mu))$.
\item Let $q_0$ have the prescribed polar matrices and zero basis
diagonal. For every nonzero $\lambda$ and every basis vector
$v$ of $R_\lambda$, impose
$\sum_{i,a}v_i\lambda_a Q_{ia}=\lambda(q_0(v))$.
\item Solve this affine linear system exactly. If inconsistent, return
an empty set; otherwise return its complete affine solution space.
For special targets also require $\bigcap_\lambda R_\lambda=0$.
\end{enumerate}
\end{algorithm}
\end{minipage}\par\medskip
This is the implementation of Theorem~\ref{thm:class2-SR-new} in
\path{sr_class2.g}; it applies to the non-special quotient inputs too,
provided the last, special-target condition is omitted.

\par\noindent\begin{minipage}{\linewidth}
\begin{algorithm}[Five-dimensional exclusion]\label{alg:small-exclusion}
\leavevmode\par
\textbf{Input:} the quotient
$E=\Alt(\mathbb F_2^5)/\langle e_{12}\rangle$ of dimension nine.
\textbf{Output:} an exhaustive test of the possible $(5,5)$ spaces.
\begin{enumerate}
\item Enumerate four-dimensional subspaces of $E$ by their unique
reduced row echelon matrices. Lift each growing row space and adjoin
$e_{12}$.
\item After adjoining a row, test all newly determined triples by
Algorithm~\ref{alg:class2-filter}, step~2. On a failure, reject this
entire branch. Count its completions from the unassigned echelon
entries, so rejected and retained branches partition all
$\binom94_2=3309747$ subspaces.
\item For each completed surviving space, solve the reality equations
in steps~3--4. Independently check whether its rank-two forms contain
a full-star hyperplane, to which Lemma~\ref{lem:star-extension} applies.
\item Check the partition sum, retain the $84$ multiplicity-free
survivors as a certificate, and check that all $84$ reality systems are
inconsistent. Return zero SR spaces.
\end{enumerate}
\end{algorithm}
\end{minipage}\par\medskip
The source is \path{mf_small.cpp}; the retained spaces are checked
independently by \path{review_small_forms.py}.

\par\noindent\begin{minipage}{\linewidth}
\begin{algorithm}[The last pencil]\label{alg:last-pencil}
\leavevmode\par
\textbf{Input:} the explicit pencil $K\perp K\perp H_{w_1}$ in the
proof of Theorem~\ref{thm:pencil-eighteen-new}.
\textbf{Output:} five inequivalent quadratic representatives.
\begin{enumerate}
\item Use Algorithm~\ref{alg:class2-filter} to construct the entire
affine set $\mathcal Q$ of $256$ refinements.
\item Construct the $15$ linear transformation pairs listed below.
For every pair $(T,C)$ check that both matrices are invertible and
$\sum_b C_{ba}TB_bT^{\mathsf T}=B_a$ for $a=1,2$.
\item Traverse the orbits of $\mathcal Q$ under
$q(v)\mapsto q(vT)C$, retaining one representative per orbit.
\item Check that the orbit sizes are $48,48,48,96,16$ and sum to
$256$. For each representative compute the pair
$(q|_X=0,\,4|q^{-1}(0)|-1)$, where $X$ is given in
\eqref{eq:lastpencil-X-new}. Check that these five pairs are distinct.
\end{enumerate}
\end{algorithm}
\end{minipage}\par\medskip
For precision, use the ordered basis $a_1,b_1,c_1,a_2,b_2,c_2,u,v$;
$u,v$ are the regular-block basis vectors denoted $p,q$ in the proof.
In the following table, an assignment specifies the image of a basis
vector, and unlisted vectors are fixed. All target matrices $C$ are
identity except in the last row.
\begin{center}\small
\begin{tabular}{@{}p{.83\linewidth}r@{}}\toprule
Transformations & Number\\\midrule
Interchange $(a_1,b_1,c_1)$ and $(a_2,b_2,c_2)$ &1\\
$a_1\mapsto a_1+a_2$, $b_2\mapsto b_2+b_1$, $c_2\mapsto c_2+c_1$ &1\\
For each $i=1,2$: $a_i\mapsto a_i+b_i$, or $a_i\mapsto a_i+c_i$ &4\\
$a_1\mapsto a_1+b_2$, $a_2\mapsto a_2+b_1$; and the analogous pair with $c$ &2\\
Interchange $u,v$; or $v\mapsto v+u$ &2\\
For each $i$: $(a_i,v)\mapsto(a_i+u,v+b_i)$; or $(a_i,u)\mapsto(a_i+v,u+b_i)$ &4\\
$c_i\mapsto c_i+b_i$ for both $i$; $C=\left(\begin{smallmatrix}1&0\\1&1\end{smallmatrix}\right)$ &1\\\bottomrule
\end{tabular}
\end{center}
The source \path{classify_last_pencil.g} checks these identities
and writes \path{last_pencil_orbits.csv}. The intrinsic-invariant
check proves that no further tensor automorphism merges its five orbits.

\par\noindent\begin{minipage}{\linewidth}
\begin{algorithm}[Quotient extensions and quadratic orbits]\label{alg:quotient-orbits}
\leavevmode\par
\textbf{Input:} $(m,r)=(6,4)$ or $(7,3)$ and the complete order-512 SR
catalogue. \textbf{Output:} one quadratic representative of every special
SR-group of order $1024$ with these dimensions.
\begin{enumerate}
\item Retain all class-two groups $P$ in the input with
$\dim(P/P')=m$. Include non-special groups. Extract their commutator
spaces $\mathcal N_P$ in $\Alt(\mathbb F_2^m)$ and remove identical
coordinate spaces by reduced row echelon form.
\item For each retained $\mathcal N_P$, visit every nonzero vector
of $\Alt(V)/\mathcal N_P$. For each extension
$\mathcal N=\mathcal N_P+\langle B\rangle$, require zero common
radical and a nonempty result from Algorithm~\ref{alg:class2-filter}.
Record all successful extensions and the exact per-parent counts.
\item Reduce the successful extensions by a verified subgroup of the
parent stabilizer. Check every generator for invertibility and
preservation of $\mathcal N_P$, and check that its orbit sizes sum
to the successful-extension count for that parent.
\item Construct the coloured polar graph of
Proposition~\ref{prop:polar-graph} for each retained extension. Use
exact graph canonical forms to retain one representative per global
polar orbit, including comparisons between different parents.
\item For each polar representative, obtain the full graph
automorphism group. Verify its point actions $(A,D)$ for linearity
and the polar identity. Solve the reality equations again and take
orbits under $q\mapsto D^{-1}\circ q\circ A$ on the entire solution space.
\item Check that the refinement-orbit sizes exhaust that space.
Retain one representative per orbit and verify their inequivalence
using canonical forms of the augmented quadratic graphs.
\end{enumerate}
\end{algorithm}
\end{minipage}\par\medskip
Theorem~\ref{thm:quotient-polar-complete} proves coverage, while
Proposition~\ref{prop:polar-graph} proves exact isomorphism control.
The implementations are \path{extend_parents.cpp},
\path{parent_subgroup_orbits.cpp}, and
\path{canonical_class2.py}. The first is an exhaustive finite loop,
and the latter steps reduce its output without discarding any orbit.

\subsection{Exporting the class-two representatives of order 1024}
\label{app:export-class2}
The following procedure constructs and checks every retained representative.

\par\noindent\begin{minipage}{\linewidth}
\begin{algorithm}[Exporting the class-two catalogue]\label{alg:class2-export}
\leavevmode\par
\noindent\textbf{Input:} the complete order-512 SR census $\mathcal C_{512}$
and one quadratic-map representative $(\beta,q)$ for each of the 126
special isomorphism types.\\
\textbf{Output:} \path{sr1024_class2.csv}, with a reconstructible pc
presentation and the stated group invariants for every class-two type.
\begin{enumerate}
\item Form the list $\mathcal P$ of class-two groups in
$\mathcal C_{512}$; check that $|\mathcal P|=95$.
\item Construct $K\times C_2$ for each $K\in\mathcal P$ and construct
$G(\beta,q)$ for each of the 126 supplied special representatives.
Retain the construction source with each group.
\item For every constructed group $G$:
  \begin{enumerate}
  \item verify $|G|=1024$, class two, the special/direct-factor flag,
  and the exact SR criterion on its extracted quadratic data;
  \item count squares of actual group elements and verify the two
  Wigner identities;
  \item compute the exponent, center and derived subgroup orders,
  conjugacy-class and involution counts, and dimensions $(m,r)$;
  \item set the parent to $C_2^m$ and the step to $r$; encode a pc
  presentation, reconstruct it, and check the reconstructed group.
  \end{enumerate}
\item Check that there are 95 non-special groups and 126 special
  groups, with the special split $28+80+18$.
\item Write all 221 CSV rows and a completion record after every check
  succeeds.
\end{enumerate}
\end{algorithm}
\end{minipage}\par\medskip
This export uses previously classified orbit representatives.  It does
not perform a new orbit enumeration or infer isomorphism from the
exported invariants.

\clearpage
\section{Character proofs for the two exceptional families}\label{app:arithmetic-families}
This appendix proves the two constructions used in
Corollary~\ref{cor:infinite-correction-main}. Related affine and
extraspecial examples appear in \cite[Chapter~XXVI, pp.~453--454]{BKZ19};
the explicit actions below fix the groups for every odd prime power $q$.

\begin{thm}[The affine norm-circle family]\label{thm:arithmetic-norm-circle}
Let $q$ be an odd prime power, let $K=\mathbb F_{q^2}$, and put
$U=\{u\in K^\times:u^{q+1}=1\}$.
Let $H=U\rtimes\langle\sigma\rangle$ act on $V=K^+$ by
$u(x)=ux$ and $\sigma(x)=x^q$.
Then $G=V\rtimes H$, of order $2q^2(q+1)$, is simply reducible.
\end{thm}
\begin{proof}
Write $N(x)=x^{q+1}$ and use the multiplication
$(v,h)(w,k)=(v+h(w),hk)$ in $G$.
First we prove ambivalence. A translation $(v,1)$ is inverted by
$(0,-1)$. If $h\in U\setminus\{1\}$, then $1-h$ is invertible on
$V$, so $(v,h)$ is conjugate under $V$ to $(0,h)$; the latter is
inverted by $(0,\sigma)$. Finally, if $h\in H\setminus U$, then
$h^2=1$, and direct multiplication gives
\[
 (v,-1)(v,h)(v,-1)^{-1}=(-h(v),h)=(v,h)^{-1}.
\]

We prove multiplicity-freeness by the little-group description of
irreducible representations of an abelian-by-finite split extension.
Identify $\widehat V$ with $K^+$ using
\[
 a\longmapsto\psi_a,\qquad
 \psi_a(v)=\exp\bigl(2\pi i\,\operatorname{Tr}_{K/\mathbb F_p}(av)/p\bigr),
 \qquad q=p^e.
\]
The contragredient action of $H$ on the labels is again the group of
norm-preserving maps generated by multiplication by $U$ and Frobenius.
Its nonzero orbits are
$\mathcal O_t=\{a:N(a)=t\}$, $t\in\mathbb F_q^\times$.
Each has size $q+1$, and the stabilizer $H_a$ is a reflection subgroup
of order two. The irreducibles of $G$ consist of those inflated from
$H$ and the representations
\[
 T_{t,\varepsilon}
 =\operatorname{Ind}_{V\rtimes H_a}^{G}
       (\psi_a\otimes\varepsilon),
 \qquad a\in\mathcal O_t,\quad
 \varepsilon\in\operatorname{Irr}(H_a).
\]
Here $\psi_a$ extends to $V\rtimes H_a$ by being trivial on $H_a$.
The restriction of $T_{t,\varepsilon}$ to $V$ contains each weight
$\psi_x$, $x\in\mathcal O_t$, exactly once.

By Proposition~\ref{prop:generalized-dihedral-sr}, the dihedral
group $H$ has multiplicity-free products of irreducibles.
Moreover, every irreducible of $H$ restricts multiplicity-freely to
each reflection subgroup: its degree is one or two, and in degree two
a reflection has eigenvalues $1,-1$. It follows at once that products
of two inflated irreducibles, or of an inflated irreducible with one
$T_{t,\varepsilon}$, are multiplicity-free.

Consider $T_{a,\varepsilon}\otimes T_{b,\delta}$, where now
$a,b\in\mathbb F_q^\times$ index norm values. At a nonzero weight
$\psi_z$ the weight space is a sum of one-dimensional spaces indexed by
\[
 \mathcal S_z=\{(x,y):x+y=z,\ N(x)=a,\ N(y)=b\}.
\]
This set has at most two members. Indeed, with $r=x/z$, its defining
conditions prescribe both
\[
 N(r)=a/N(z),\qquad
 r+r^q=1+(a-b)/N(z).
\]
Thus $r$ is a root of a fixed quadratic polynomial over $\mathbb F_q$.
If the prescribed trace and norm have two solutions, they are
exchanged by Frobenius. Indeed, a solution in $\mathbb F_q$ would
equal half the prescribed trace and give a double root. Equivalently,
the nonidentity element of $H_z$ exchanges the two members of
$\mathcal S_z$. Its action on the corresponding two-dimensional
weight space therefore has trace zero and square one, and contains
each of the two characters of $H_z$ once. If there is at most one
member, the same multiplicity bound is immediate. The little-group
correspondence now shows that every irreducible supported on a
nonzero $V$-orbit occurs at most once in the tensor product.

The zero-weight space is zero unless $a=b$. When $a=b$, its labels are
$(x,-x)$ with $N(x)=a$, a transitive $H$-set whose stabilizer is
$H_x\cong C_2$. The corresponding representation of $H$ is
$\operatorname{Ind}_{H_x}^{H}\eta$ for some linear character $\eta$
of $H_x$. Frobenius reciprocity and the restriction property of
dihedral irreducibles proved above show that this representation is
multiplicity-free. It gives exactly the inflated constituents of
the tensor product. All cases have been treated, so $G$ is SR.
\end{proof}

\begin{thm}[A Heisenberg extension of the norm-circle family]
\label{thm:arithmetic-heisenberg-circle}
Let $q$ be an odd prime power, let $V$ be a two-dimensional
$\mathbb F_q$-space, and let $\omega$ be a nondegenerate alternating
form on $V$. Choose a nonsplit norm form on $V$ and let
$D=O^-_2(q)=U\rtimes\langle\sigma\rangle$, where $|U|=q+1$ and
$\sigma$ is a reflection. Define
\[
 P=V\times\mathbb F_q,\qquad
 (v,z)(w,s)=(v+w,z+s+\tfrac12\omega(v,w)).
\]
The action $d(v,z)=(dv,\det(d)z)$ defines a group
$\widetilde G_q=P\rtimes D$ of order $2q^3(q+1)$.
Then $\widetilde G_q$ is simply reducible.
\end{thm}
\begin{proof}
Write $J=P\rtimes U$ and $Z=Z(P)$.
We first verify ambivalence. Conjugation within $P$ changes the
central coordinate of $(v,z)$ arbitrarily if $v\ne0$; the rotation
$-I$ changes $v$ to $-v$. A pure central element is inverted by
$\sigma$. If $u\in U\setminus\{1\}$, then $1-u$ is invertible on
$V$, so every element with rotational part $u$ is $P$-conjugate to
$(0,z;u)$. Reflection by $\sigma$ inverts this representative.
Finally let $r\in D\setminus U$. Conjugation by a translation
removes the $(-1)$-eigenspace component of $v$, and conjugation by
a central element removes the central coordinate because $r$
negates $Z$. Thus an element with reflection part $r$ is conjugate
to $(v_+,0;r)$, where $rv_+=v_+$. The rotation $-I$ inverts this
representative.

Fix a nontrivial additive character $\psi$ of $\mathbb F_q$.
The irreducible characters of $P$ consist of the $q^2$ linear
characters indexed by $V$ and, for each $\lambda\in\mathbb F_q^\times$,
a unique character $\rho_\lambda$ of degree $q$ with central
character $z\mapsto\psi(\lambda z)$. Each $\rho_\lambda$ extends
to $J$: an intertwiner for a generator of the cyclic group $U$
can be rescaled to have order $q+1$. Fix such extensions.

We require the following elementary character identity:
\begin{equation}\label{eq:circle-schrodinger-trace}
 |\rho_\lambda(u)|^2=1
 \quad(1\ne u\in U).
\end{equation}
Indeed, the $q^2$ Weyl operators
$\rho_\lambda((v,0))$, $v\in V$, form a basis of
$\operatorname{End}(\rho_\lambda)$. Conjugation by the extension
of $u$ permutes this basis by $v\mapsto uv$. A nonidentity norm
rotation fixes only $v=0$, so the trace on the endomorphism space
is one. This trace equals $|\rho_\lambda(u)|^2$.
In particular,
\[
 \langle(\rho_\lambda)_U,(\rho_\lambda)_U\rangle_U
 =\frac{q^2+q}{q+1}=q.
\]
Because $U$ is cyclic and $\rho_\lambda(1)=q$, its restriction to
$U$ is multiplicity-free.

The irreducibles of $J$ are: the linear characters inflated from
$U$; the characters $L_t$ of degree $q+1$ induced from the nonzero
norm orbits of linear characters of $P$, with
$t\in\mathbb F_q^\times$; and
$R_{\lambda,\nu}=\rho_\lambda\nu$, where
$\lambda\ne0$ and $\nu\in\widehat U$.
Reflection fixes each isomorphism type $L_t$ and changes the
central character $\lambda$ to $-\lambda$.
Consequently the irreducibles of $\widetilde G_q$ not having $Z$
in their kernel are precisely
\[
 T_{\lambda,\nu}
 =\operatorname{Ind}_{J}^{\widetilde G_q}R_{\lambda,\nu},
\]
with one representative for each reflection orbit; they have
degree $2q$. The other irreducibles inflate from
$\widetilde G_q/Z$, which is the affine norm-circle group of
Theorem~\ref{thm:arithmetic-norm-circle}.

Products of two such inflated irreducibles are already
multiplicity-free. If $A$ is inflated and $T=\operatorname{Ind}R$,
the projection formula gives
$A T=\operatorname{Ind}(A_JR)$. When $A$ comes from the dihedral
quotient, $A_J$ is one linear $U$-character or two distinct such
characters. Their products with $R$ remain distinct, also after
induction, because they have the same nonzero central character
and reflection changes its sign.
In the other case $A_J=L_t$. The $P$-restriction of $RL_t$ is
$(q+1)\rho_\lambda$, so
$RL_t=\rho_\lambda\otimes M$ for a $U$-module $M$ of dimension
$q+1$. Since $L_t(u)=0$ for $u\ne1$, identity
\eqref{eq:circle-schrodinger-trace} shows that $M(u)=0$ there.
Thus $M$ is the regular representation of $U$, proving
multiplicity-freeness in this case as well.

It remains to multiply two characters $T$. The index-two
induction formula gives
\[
 (\operatorname{Ind}R_\lambda)(\operatorname{Ind}R_\mu)
 =\operatorname{Ind}(R_\lambda R_\mu)
  +\operatorname{Ind}(R_\lambda R_\mu^{\sigma}).
\]
Their central characters are $\lambda+\mu$ and $\lambda-\mu$;
the nonzero sign-orbits of these two parameters are distinct.
If $\alpha+\beta\ne0$, the $P$-restriction of
$R_\alpha R_\beta$ is $q\rho_{\alpha+\beta}$. Therefore
\[
 R_\alpha R_\beta=\rho_{\alpha+\beta}\otimes M,
 \qquad \dim M=q,
\]
for a $U$-module $M$. Equation~\eqref{eq:circle-schrodinger-trace}
gives $|M(u)|=1$ for every $u\ne1$. Hence
$\langle M,M\rangle_U=(q^2+q)/(q+1)=q=\dim M$, so $M$ is
multiplicity-free. Induction again cannot identify distinct
constituents with the same nonzero central character.

If $\alpha+\beta=0$, the $P$-restriction of
$R_\alpha R_\beta$ is the sum of every linear character of $P$
once. As a $J$-representation it is therefore one linear
$U$-character together with every $L_t$ once. The linear character
induces multiplicity-freely to the dihedral quotient, and each
$L_t$ induces to the sum of its two distinct extensions.
This case is multiplicity-free and its constituents, having $Z$
in their kernel, cannot coincide with those from the nonzero
central-character term. All tensor products have been treated.
\end{proof}

\section{Implementation and reproducibility}\label{app:implementation}
The accompanying archive contains the source programs, input catalogues,
quadratic representatives and completed execution certificates. The
following mapping connects the pseudocode to the executable sources.
\begin{center}\small
\begin{tabular}{@{}>{\raggedright\arraybackslash}p{.29\linewidth}>{\raggedright\arraybackslash}p{.64\linewidth}@{}}\toprule
Procedure & Source files\\\midrule
Algorithms~\ref{alg:srtest}--\ref{alg:catalogue}
 & \path{sr_search2_2000.g}, \path{sr_search512.g},
   \path{sr_search1536.g}\\
Algorithm~\ref{alg:descendants}
 & \path{scriptG.g}; completed checkpoint and console transcript\\
Algorithm~\ref{alg:exportdescendants}
 & \path{makeCSV.g}\\
Algorithm~\ref{alg:class2-filter}
 & \path{sr_class2.g}\\
Algorithm~\ref{alg:small-exclusion}
 & \path{mf_small.cpp}, \path{review_small_forms.py}\\
Algorithm~\ref{alg:last-pencil}
 & \path{classify_last_pencil.g}\\
Algorithm~\ref{alg:quotient-orbits}
 & \path{extend_parents.cpp}, \path{parent_subgroup_orbits.cpp},
   \path{canonical_class2.py}\\
Algorithm~\ref{alg:class2-export}
 & \path{export_sr1024_class2.g},
   \path{sr1024_special_representatives.g}\\
CSV reconstruction
 & \path{read_sr1024_class3plus.g}\\
Arithmetic checks in Section~\ref{sec:count-patterns}
 & \path{counting_structures.py}, \path{abelian_hall_counts.py},
   \path{verify_progressions.py}\\\bottomrule
\end{tabular}
\end{center}
The README gives execution order and dependencies. Software versions are specified in the references and recorded in
execution logs. The source archive contains the completed descendant
checkpoint and console transcript together with the programs that
generate, audit and export their records.

The core class-two implementation was checked by $79$ exact ordinary
character comparisons and $79$ extraction--reconstruction tests at smaller
orders. The three known order-512 groups discussed after
Theorem~\ref{thm:class2-complete} test the distinction between polar and
quadratic equivalence. Every new special representative was checked by
the form criterion, independent Wigner moments and presentation
reconstruction. The separate class-two exporter performs these checks
on all $221$ representatives and requires the $95+126$ split before
writing its CSV. The combined catalogue
\path{SR_groups_results1024.csv} contains all $803$ representatives;
its exporter independently verifies the Wigner identities and direct
$C_2$-factor flags, recovering $317$ positive flags.

The canonical-descendant checkpoint audit checks every expected job,
including zero-output jobs, and matches its $581$ presentation codes
with the CSV. This checks the original completed computation; it does
not repeat the large descendant-generation run. Its completeness and
nonredundancy use the generation theorem in Section~\ref{sec:1024higher}.
The source \path{audit_count_patterns.g} checks the Hall-subgroup
classification used in Corollary~\ref{cor:four-odd-census}; its input list
is included. No comparison of numerical group invariants is used as an
isomorphism test in the classification.

\subsection{Reconstructing a row and checking the arithmetic}
\label{app:reconstruction-usage}
The programs, data, and execution records are arranged in the
\path{code/}, \path{data/}, and \path{validation/} directories of the
accompanying reproduction package. From its root, the following GAP
commands reconstruct the first high-class representative and check its
order, class, and SR property:
\begin{lstlisting}
Read("code/read_sr1024_class3plus.g");;
rows := SR1024ReadClass3Plus("data/sr1024_class3plus.csv");;
G := SR1024GroupFromClass3Row(rows[1]);;
[Size(G), NilpotencyClassOfGroup(G), SR1024CheckSR(G)];
# [ 1024, 3, true ]
\end{lstlisting}
Replacing the index $1$ selects any other row. The complete catalogue
\path{SR_groups_results1024.csv} and the class-two catalogue use the
same exact-integer \texttt{Code} convention; the paper's orbit arguments
establish nonisomorphism of their representatives.

The arithmetic verification uses the accompanying census as input. For
example, from the same package root, run
\begin{lstlisting}
python3 code/counting_structures.py data/f_of_n.csv \
  --class2-audit code/audited_class2.csv \
  --class2-1024 data/sr1024_class2.csv \
  --output-dir validation
python3 code/abelian_hall_counts.py data/f_of_n.csv \
  --output validation/abelian_hall_counts.json
python3 code/verify_progressions.py data/f_of_n.csv \
  --output validation/progression_checks.json
\end{lstlisting}
The first program reconstructs the census from the directly
indecomposable counts and independently checks the formal logarithm
identity. The second computes the dihedral-product subclass by a
separate multiset calculation and compares every coefficient with the
cycle-index recurrence. The third checks the exact formulas at every
applicable order in the census. The JSON certificates record input
hashes and the resulting counts. These arithmetic calculations use
Python with its standard library and do not repeat the underlying
group enumeration.

\section*{Acknowledgements}
The author thanks Prof. Lizhong Wang, Prof. Ping Jin, and Prof. Yanjun Liu
for helpful discussions and encouragement concerning the study of group theory.
The author gratefully acknowledges the School of Mathematics and the
Informatization Office of Shandong University for providing server computing resources.

\section*{AI disclosure}
The author used AI to write GAP, Python, and C++ code implementing the author's algorithm, to explore proofs, and to assist in drafting and checking references. All mathematical ideas, algorithmic concepts, conceptual framework, and proof strategies were developed by the author. The author has verified all AI-assisted arguments and citations and takes full responsibility for the content of the paper. The author gratefully acknowledges the help of the different AI models used in this work.

\section*{Funding}
The author is supported by the Shandong Provincial Natural Science Foundation (Grant No. ZR2025MS36).
\bibliographystyle{amsplain}
\bibliography{SRgroup2000}
\end{document}

%% file: f_table_body.tex
2 & 1 & 4 & 1 & 6 & 1 & 8 & 3 & 10 & 1 \\
12 & 1 & 14 & 1 & 16 & 5 & 18 & 2 & 20 & 1 \\
22 & 1 & 24 & 4 & 26 & 1 & 28 & 1 & 30 & 1 \\
32 & 12 & 34 & 1 & 36 & 3 & 38 & 1 & 40 & 3 \\
42 & 1 & 44 & 1 & 46 & 1 & 48 & 8 & 50 & 2 \\
52 & 1 & 54 & 3 & 56 & 3 & 58 & 1 & 60 & 2 \\
62 & 1 & 64 & 25 & 66 & 1 & 68 & 1 & 70 & 1 \\
72 & 8 & 74 & 1 & 76 & 1 & 78 & 1 & 80 & 7 \\
82 & 1 & 84 & 2 & 86 & 1 & 88 & 3 & 90 & 2 \\
92 & 1 & 94 & 1 & 96 & 18 & 98 & 2 & 100 & 3 \\
102 & 1 & 104 & 3 & 106 & 1 & 108 & 5 & 110 & 1 \\
112 & 7 & 114 & 1 & 116 & 1 & 118 & 1 & 120 & 4 \\
122 & 1 & 124 & 1 & 126 & 2 & 128 & 54 & 130 & 1 \\
132 & 2 & 134 & 1 & 136 & 3 & 138 & 1 & 140 & 2 \\
142 & 1 & 144 & 20 & 146 & 1 & 148 & 1 & 150 & 2 \\
152 & 3 & 154 & 1 & 156 & 2 & 158 & 1 & 160 & 16 \\
162 & 5 & 164 & 1 & 166 & 1 & 168 & 4 & 170 & 1 \\
172 & 1 & 174 & 1 & 176 & 7 & 178 & 1 & 180 & 5 \\
182 & 1 & 184 & 3 & 186 & 1 & 188 & 1 & 190 & 1 \\
192 & 42 & 194 & 1 & 196 & 3 & 198 & 2 & 200 & 7 \\
202 & 1 & 204 & 2 & 206 & 1 & 208 & 7 & 210 & 1 \\
212 & 1 & 214 & 1 & 216 & 13 & 218 & 1 & 220 & 2 \\
222 & 1 & 224 & 15 & 226 & 1 & 228 & 2 & 230 & 1 \\
232 & 3 & 234 & 2 & 236 & 1 & 238 & 1 & 240 & 13 \\
242 & 2 & 244 & 1 & 246 & 1 & 248 & 3 & 250 & 3 \\
252 & 5 & 254 & 1 & 256 & 130 & 258 & 1 & 260 & 2 \\
262 & 1 & 264 & 4 & 266 & 1 & 268 & 1 & 270 & 3 \\
272 & 7 & 274 & 1 & 276 & 2 & 278 & 1 & 280 & 4 \\
282 & 1 & 284 & 1 & 286 & 1 & 288 & 42 & 290 & 1 \\
292 & 1 & 294 & 2 & 296 & 3 & 298 & 1 & 300 & 6 \\
302 & 1 & 304 & 7 & 306 & 2 & 308 & 2 & 310 & 1 \\
312 & 4 & 314 & 1 & 316 & 1 & 318 & 1 & 320 & 36 \\
322 & 1 & 324 & 11 & 326 & 1 & 328 & 3 & 330 & 1 \\
332 & 1 & 334 & 1 & 336 & 13 & 338 & 2 & 340 & 2 \\
342 & 2 & 344 & 3 & 346 & 1 & 348 & 2 & 350 & 2 \\
352 & 15 & 354 & 1 & 356 & 1 & 358 & 1 & 360 & 10 \\
362 & 1 & 364 & 2 & 366 & 1 & 368 & 7 & 370 & 1 \\
372 & 2 & 374 & 1 & 376 & 3 & 378 & 3 & 380 & 2 \\
382 & 1 & 384 & 94 & 386 & 1 & 388 & 1 & 390 & 1 \\
392 & 7 & 394 & 1 & 396 & 5 & 398 & 1 & 400 & 17 \\
402 & 1 & 404 & 1 & 406 & 1 & 408 & 4 & 410 & 1 \\
412 & 1 & 414 & 2 & 416 & 15 & 418 & 1 & 420 & 4 \\
422 & 1 & 424 & 3 & 426 & 1 & 428 & 1 & 430 & 1 \\
432 & 37 & 434 & 1 & 436 & 1 & 438 & 1 & 440 & 4 \\
442 & 1 & 444 & 2 & 446 & 1 & 448 & 35 & 450 & 4 \\
452 & 1 & 454 & 1 & 456 & 4 & 458 & 1 & 460 & 2 \\
462 & 1 & 464 & 7 & 466 & 1 & 468 & 5 & 470 & 1 \\
472 & 3 & 474 & 1 & 476 & 2 & 478 & 1 & 480 & 29 \\
482 & 1 & 484 & 3 & 486 & 7 & 488 & 3 & 490 & 2 \\
492 & 2 & 494 & 1 & 496 & 7 & 498 & 1 & 500 & 5 \\
502 & 1 & 504 & 10 & 506 & 1 & 508 & 1 & 510 & 1 \\
\textbf{512} & \textbf{317} & 514 & 1 & 516 & 2 & 518 & 1 & 520 & 4 \\
522 & 2 & 524 & 1 & 526 & 1 & 528 & 13 & 530 & 1 \\
532 & 2 & 534 & 1 & 536 & 3 & 538 & 1 & 540 & 10 \\
542 & 1 & 544 & 15 & 546 & 1 & 548 & 1 & 550 & 2 \\
552 & 4 & 554 & 1 & 556 & 1 & 558 & 2 & 560 & 12 \\
562 & 1 & 564 & 2 & 566 & 1 & 568 & 3 & 570 & 1 \\
572 & 2 & 574 & 1 & 576 & 104 & 578 & 2 & 580 & 2 \\
582 & 1 & 584 & 3 & 586 & 1 & 588 & 5 & 590 & 1 \\
592 & 7 & 594 & 3 & 596 & 1 & 598 & 1 & 600 & 12 \\
602 & 1 & 604 & 1 & 606 & 1 & 608 & 15 & 610 & 1 \\
612 & 5 & 614 & 1 & 616 & 4 & 618 & 1 & 620 & 2 \\
622 & 1 & 624 & 13 & 626 & 1 & 628 & 1 & 630 & 2 \\
632 & 3 & 634 & 1 & 636 & 2 & 638 & 1 & 640 & 84 \\
642 & 1 & 644 & 2 & 646 & 1 & 648 & 23 & 650 & 2 \\
652 & 1 & 654 & 1 & 656 & 7 & 658 & 1 & 660 & 4 \\
662 & 1 & 664 & 3 & 666 & 2 & 668 & 1 & 670 & 1 \\
672 & 29 & 674 & 1 & 676 & 3 & 678 & 1 & 680 & 4 \\
682 & 1 & 684 & 5 & 686 & 3 & 688 & 7 & 690 & 1 \\
692 & 1 & 694 & 1 & 696 & 4 & 698 & 1 & 700 & 5 \\
702 & 3 & 704 & 35 & 706 & 1 & 708 & 2 & 710 & 1 \\
712 & 3 & 714 & 1 & 716 & 1 & 718 & 1 & 720 & 32 \\
722 & 2 & 724 & 1 & 726 & 2 & 728 & 4 & 730 & 1 \\
732 & 2 & 734 & 1 & 736 & 15 & 738 & 2 & 740 & 2 \\
742 & 1 & 744 & 4 & 746 & 1 & 748 & 2 & 750 & 3 \\
752 & 7 & 754 & 1 & 756 & 10 & 758 & 1 & 760 & 4 \\
762 & 1 & 764 & 1 & 766 & 1 & 768 & 227 & 770 & 1 \\
772 & 1 & 774 & 2 & 776 & 3 & 778 & 1 & 780 & 4 \\
782 & 1 & 784 & 18 & 786 & 1 & 788 & 1 & 790 & 1 \\
792 & 10 & 794 & 1 & 796 & 1 & 798 & 1 & 800 & 39 \\
802 & 1 & 804 & 2 & 806 & 1 & 808 & 3 & 810 & 5 \\
812 & 2 & 814 & 1 & 816 & 13 & 818 & 1 & 820 & 2 \\
822 & 1 & 824 & 3 & 826 & 1 & 828 & 5 & 830 & 1 \\
832 & 35 & 834 & 1 & 836 & 2 & 838 & 1 & 840 & 7 \\
842 & 1 & 844 & 1 & 846 & 2 & 848 & 7 & 850 & 2 \\
852 & 2 & 854 & 1 & 856 & 3 & 858 & 1 & 860 & 2 \\
862 & 1 & 864 & 81 & 866 & 1 & 868 & 2 & 870 & 1 \\
872 & 3 & 874 & 1 & 876 & 2 & 878 & 1 & 880 & 12 \\
882 & 4 & 884 & 2 & 886 & 1 & 888 & 4 & 890 & 1 \\
892 & 1 & 894 & 1 & 896 & 81 & 898 & 1 & 900 & 13 \\
902 & 1 & 904 & 3 & 906 & 1 & 908 & 1 & 910 & 1 \\
912 & 13 & 914 & 1 & 916 & 1 & 918 & 3 & 920 & 4 \\
922 & 1 & 924 & 4 & 926 & 1 & 928 & 15 & 930 & 1 \\
932 & 1 & 934 & 1 & 936 & 10 & 938 & 1 & 940 & 2 \\
942 & 1 & 944 & 7 & 946 & 1 & 948 & 2 & 950 & 2 \\
952 & 4 & 954 & 2 & 956 & 1 & 958 & 1 & 960 & 76 \\
962 & 1 & 964 & 1 & 966 & 1 & 968 & 7 & 970 & 1 \\
972 & 18 & 974 & 1 & 976 & 7 & 978 & 1 & 980 & 5 \\
982 & 1 & 984 & 4 & 986 & 1 & 988 & 2 & 990 & 2 \\
992 & 15 & 994 & 1 & 996 & 2 & 998 & 1 & 1000 & 12 \\
1002 & 1 & 1004 & 1 & 1006 & 1 & 1008 & 32 & 1010 & 1 \\
1012 & 2 & 1014 & 2 & 1016 & 3 & 1018 & 1 & 1020 & 4 \\
1022 & 1 & \textbf{1024} & \textbf{803} & 1026 & 3 & 1028 & 1 & 1030 & 1 \\
1032 & 4 & 1034 & 1 & 1036 & 2 & 1038 & 1 & 1040 & 12 \\
1042 & 1 & 1044 & 5 & 1046 & 1 & 1048 & 3 & 1050 & 2 \\
1052 & 1 & 1054 & 1 & 1056 & 29 & 1058 & 2 & 1060 & 2 \\
1062 & 2 & 1064 & 4 & 1066 & 1 & 1068 & 2 & 1070 & 1 \\
1072 & 7 & 1074 & 1 & 1076 & 1 & 1078 & 2 & 1080 & 19 \\
1082 & 1 & 1084 & 1 & 1086 & 1 & 1088 & 35 & 1090 & 1 \\
1092 & 4 & 1094 & 1 & 1096 & 3 & 1098 & 2 & 1100 & 5 \\
1102 & 1 & 1104 & 13 & 1106 & 1 & 1108 & 1 & 1110 & 1 \\
1112 & 3 & 1114 & 1 & 1116 & 5 & 1118 & 1 & 1120 & 28 \\
1122 & 1 & 1124 & 1 & 1126 & 1 & 1128 & 4 & 1130 & 1 \\
1132 & 1 & 1134 & 5 & 1136 & 7 & 1138 & 1 & 1140 & 4 \\
1142 & 1 & 1144 & 4 & 1146 & 1 & 1148 & 2 & 1150 & 2 \\
1152 & 244 & 1154 & 1 & 1156 & 3 & 1158 & 1 & 1160 & 4 \\
1162 & 1 & 1164 & 2 & 1166 & 1 & 1168 & 7 & 1170 & 2 \\
1172 & 1 & 1174 & 1 & 1176 & 10 & 1178 & 1 & 1180 & 2 \\
1182 & 1 & 1184 & 15 & 1186 & 1 & 1188 & 10 & 1190 & 1 \\
1192 & 3 & 1194 & 1 & 1196 & 2 & 1198 & 1 & 1200 & 34 \\
1202 & 1 & 1204 & 2 & 1206 & 2 & 1208 & 3 & 1210 & 2 \\
1212 & 2 & 1214 & 1 & 1216 & 35 & 1218 & 1 & 1220 & 2 \\
1222 & 1 & 1224 & 10 & 1226 & 1 & 1228 & 1 & 1230 & 1 \\
1232 & 12 & 1234 & 1 & 1236 & 2 & 1238 & 1 & 1240 & 4 \\
1242 & 3 & 1244 & 1 & 1246 & 1 & 1248 & 29 & 1250 & 5 \\
1252 & 1 & 1254 & 1 & 1256 & 3 & 1258 & 1 & 1260 & 10 \\
1262 & 1 & 1264 & 7 & 1266 & 1 & 1268 & 1 & 1270 & 1 \\
1272 & 4 & 1274 & 2 & 1276 & 2 & 1278 & 2 & 1280 & 202 \\
1282 & 1 & 1284 & 2 & 1286 & 1 & 1288 & 4 & 1290 & 1 \\
1292 & 2 & 1294 & 1 & 1296 & 70 & 1298 & 1 & 1300 & 5 \\
1302 & 1 & 1304 & 3 & 1306 & 1 & 1308 & 2 & 1310 & 1 \\
1312 & 15 & 1314 & 2 & 1316 & 2 & 1318 & 1 & 1320 & 7 \\
1322 & 1 & 1324 & 1 & 1326 & 1 & 1328 & 7 & 1330 & 1 \\
1332 & 5 & 1334 & 1 & 1336 & 3 & 1338 & 1 & 1340 & 2 \\
1342 & 1 & 1344 & 75 & 1346 & 1 & 1348 & 1 & 1350 & 6 \\
1352 & 7 & 1354 & 1 & 1356 & 2 & 1358 & 1 & 1360 & 12 \\
1362 & 1 & 1364 & 2 & 1366 & 1 & 1368 & 10 & 1370 & 1 \\
1372 & 5 & 1374 & 1 & 1376 & 15 & 1378 & 1 & 1380 & 4 \\
1382 & 1 & 1384 & 3 & 1386 & 2 & 1388 & 1 & 1390 & 1 \\
1392 & 13 & 1394 & 1 & 1396 & 1 & 1398 & 1 & 1400 & 10 \\
1402 & 1 & 1404 & 10 & 1406 & 1 & 1408 & 81 & 1410 & 1 \\
1412 & 1 & 1414 & 1 & 1416 & 4 & 1418 & 1 & 1420 & 2 \\
1422 & 2 & 1424 & 7 & 1426 & 1 & 1428 & 4 & 1430 & 1 \\
1432 & 3 & 1434 & 1 & 1436 & 1 & 1438 & 1 & 1440 & 78 \\
1442 & 1 & 1444 & 3 & 1446 & 1 & 1448 & 3 & 1450 & 2 \\
1452 & 5 & 1454 & 1 & 1456 & 12 & 1458 & 11 & 1460 & 2 \\
1462 & 1 & 1464 & 4 & 1466 & 1 & 1468 & 1 & 1470 & 2 \\
1472 & 35 & 1474 & 1 & 1476 & 5 & 1478 & 1 & 1480 & 4 \\
1482 & 1 & 1484 & 2 & 1486 & 1 & 1488 & 13 & 1490 & 1 \\
1492 & 1 & 1494 & 2 & 1496 & 4 & 1498 & 1 & 1500 & 11 \\
1502 & 1 & 1504 & 15 & 1506 & 1 & 1508 & 2 & 1510 & 1 \\
1512 & 19 & 1514 & 1 & 1516 & 1 & 1518 & 1 & 1520 & 12 \\
1522 & 1 & 1524 & 2 & 1526 & 1 & 1528 & 3 & 1530 & 2 \\
1532 & 1 & 1534 & 1 & \textbf{1536} & \textbf{559} & 1538 & 1 & 1540 & 4 \\
1542 & 1 & 1544 & 3 & 1546 & 1 & 1548 & 5 & 1550 & 2 \\
1552 & 7 & 1554 & 1 & 1556 & 1 & 1558 & 1 & 1560 & 7 \\
1562 & 1 & 1564 & 2 & 1566 & 3 & 1568 & 41 & 1570 & 1 \\
1572 & 2 & 1574 & 1 & 1576 & 3 & 1578 & 1 & 1580 & 2 \\
1582 & 1 & 1584 & 32 & 1586 & 1 & 1588 & 1 & 1590 & 1 \\
1592 & 3 & 1594 & 1 & 1596 & 4 & 1598 & 1 & 1600 & 93 \\
1602 & 2 & 1604 & 1 & 1606 & 1 & 1608 & 4 & 1610 & 1 \\
1612 & 2 & 1614 & 1 & 1616 & 7 & 1618 & 1 & 1620 & 21 \\
1622 & 1 & 1624 & 4 & 1626 & 1 & 1628 & 2 & 1630 & 1 \\
1632 & 29 & 1634 & 1 & 1636 & 1 & 1638 & 2 & 1640 & 4 \\
1642 & 1 & 1644 & 2 & 1646 & 1 & 1648 & 7 & 1650 & 2 \\
1652 & 2 & 1654 & 1 & 1656 & 10 & 1658 & 1 & 1660 & 2 \\
1662 & 1 & 1664 & 81 & 1666 & 2 & 1668 & 2 & 1670 & 1 \\
1672 & 4 & 1674 & 3 & 1676 & 1 & 1678 & 1 & 1680 & 24 \\
1682 & 2 & 1684 & 1 & 1686 & 1 & 1688 & 3 & 1690 & 2 \\
1692 & 5 & 1694 & 2 & 1696 & 15 & 1698 & 1 & 1700 & 5 \\
1702 & 1 & 1704 & 4 & 1706 & 1 & 1708 & 2 & 1710 & 2 \\
1712 & 7 & 1714 & 1 & 1716 & 4 & 1718 & 1 & 1720 & 4 \\
1722 & 1 & 1724 & 1 & 1726 & 1 & 1728 & 208 & 1730 & 1 \\
1732 & 1 & 1734 & 2 & 1736 & 4 & 1738 & 1 & 1740 & 4 \\
1742 & 1 & 1744 & 7 & 1746 & 2 & 1748 & 2 & 1750 & 3 \\
1752 & 4 & 1754 & 1 & 1756 & 1 & 1758 & 1 & 1760 & 28 \\
1762 & 1 & 1764 & 13 & 1766 & 1 & 1768 & 4 & 1770 & 1 \\
1772 & 1 & 1774 & 1 & 1776 & 13 & 1778 & 1 & 1780 & 2 \\
1782 & 5 & 1784 & 3 & 1786 & 1 & 1788 & 2 & 1790 & 1 \\
1792 & 197 & 1794 & 1 & 1796 & 1 & 1798 & 1 & 1800 & 27 \\
1802 & 1 & 1804 & 2 & 1806 & 1 & 1808 & 7 & 1810 & 1 \\
1812 & 2 & 1814 & 1 & 1816 & 3 & 1818 & 2 & 1820 & 4 \\
1822 & 1 & 1824 & 29 & 1826 & 1 & 1828 & 1 & 1830 & 1 \\
1832 & 3 & 1834 & 1 & 1836 & 10 & 1838 & 1 & 1840 & 12 \\
1842 & 1 & 1844 & 1 & 1846 & 1 & 1848 & 7 & 1850 & 2 \\
1852 & 1 & 1854 & 2 & 1856 & 35 & 1858 & 1 & 1860 & 4 \\
1862 & 2 & 1864 & 3 & 1866 & 1 & 1868 & 1 & 1870 & 1 \\
1872 & 32 & 1874 & 1 & 1876 & 2 & 1878 & 1 & 1880 & 4 \\
1882 & 1 & 1884 & 2 & 1886 & 1 & 1888 & 15 & 1890 & 3 \\
1892 & 2 & 1894 & 1 & 1896 & 4 & 1898 & 1 & 1900 & 5 \\
1902 & 1 & 1904 & 12 & 1906 & 1 & 1908 & 5 & 1910 & 1 \\
1912 & 3 & 1914 & 1 & 1916 & 1 & 1918 & 1 & 1920 & 185 \\
1922 & 2 & 1924 & 2 & 1926 & 2 & 1928 & 3 & 1930 & 1 \\
1932 & 4 & 1934 & 1 & 1936 & 17 & 1938 & 1 & 1940 & 2 \\
1942 & 1 & 1944 & 40 & 1946 & 1 & 1948 & 1 & 1950 & 2 \\
1952 & 15 & 1954 & 1 & 1956 & 2 & 1958 & 1 & 1960 & 10 \\
1962 & 2 & 1964 & 1 & 1966 & 1 & 1968 & 13 & 1970 & 1 \\
1972 & 2 & 1974 & 1 & 1976 & 4 & 1978 & 1 & 1980 & 10 \\
1982 & 1 & 1984 & 35 & 1986 & 1 & 1988 & 2 & 1990 & 1 \\
1992 & 4 & 1994 & 1 & 1996 & 1 & 1998 & 3 & 2000 & 32 \\